\documentclass{amsart}

\usepackage{amsfonts,amsmath,amssymb,pdfsync,color,graphicx}
\usepackage{amsthm}
\usepackage{color}
\usepackage{multicol}
\usepackage[normalem]{ulem}
		
\usepackage[english]{babel}

\usepackage{tikz}

\usepackage[T1]{fontenc}
\usepackage[latin1]{inputenc}

\usepackage{hyperref}

\newcommand{\Z}{\mathbb{Z}}
\newcommand{\R}{\mathbb{R}}
\newcommand{\N}{\mathbb{N}}
\newcommand{\T}{\mathbb{T}}
\newcommand{\vertices}{\operatorname{vert}}
\newcommand{\length}{\operatorname{length}}

\usepackage{appendix}
\usepackage{multirow}

\newtheorem{theorem}{Theorem}[section]
\newtheorem{proposition}[theorem]{Proposition}
\newtheorem{lemma}[theorem]{Lemma}
\newtheorem{corollary}[theorem]{Corollary}
\newtheorem{remark}[theorem]{Remark}

\theoremstyle{definition}
\newtheorem{definition}[theorem]{Definition}

\newcommand{\vol}{\operatorname{vol}} 
\newcommand{\Vol}{\operatorname{Vol}} 
\newcommand{\conv}{\operatorname{conv}}

\title{Classification of empty lattice $4$-simplices of width larger than two}

\author{Óscar Iglesias-Valiño \and Francisco Santos}
\address[O.~Iglesias-Valiño and F.~Santos]
{Department of Mathematics, Statistics and Computer Science. University of Cantabria, SPAIN\\
oscar.iglesias@unican.es, francisco.santos@unican.es}

\thanks{Supported by grants MTM2014-54207-P (both authors) and BES-2015-073128 (O.Iglesias) of the Spanish Ministry of Economy and Competitiveness. F.~Santos is also supported by an Einstein Visiting Professorship of the Einstein Foundation Berlin}

\begin{document}

\subjclass[2010]{Primary 52B20; Secondary 11H06.}
\keywords{Lattice polytopes, empty, classification, simplices, dimension 4, lattice width, maximum volume, enumeration.}

\begin{abstract}

A lattice $d$-simplex is the convex hull of $d+1$ affinely independent integer points in ${\mathbb R}^d$. It is called \emph{empty} if it contains no lattice point apart of its $d+1$ vertices. The classification of empty $3$-simplices is known since 1964 (White), based on the fact that they all have width one. But for dimension $4$ no complete classification is known.

Haase and Ziegler (2000) enumerated all empty $4$-simplices up to determinant 1000 and based on their results conjectured that after determinant $179$ all empty $4$-simplices have width one or two. We prove this conjecture as follows:

- We show that no empty $4$-simplex of width three or more can have determinant greater than 5058, by combining the recent classification of hollow 3-polytopes (Averkov, Kr\"umpelmann and Weltge, 2017) with general methods from the geometry of numbers.

- We continue the computations of Haase and Ziegler up to determinant 7600, and find that no new $4$-simplices of width larger than two arise. 

In particular, we give the whole list of empty $4$-simplices of width larger than two, which is as computed by Haase and Ziegler: There is a single empty $4$-simplex of width four (of determinant 101), and 178 empty $4$-simplices of width three, with determinants ranging from 41 to 179. 
\end{abstract}

\maketitle

\section{Introduction}

A \emph{lattice $d$-simplex} is the convex hull of $d+1$ affinely independent integer points in ${\mathbb R}^d$. It is called \emph{empty} if it contains no lattice point apart of the $d+1$ vertices. Empty simplices are the fundamental building blocks in the theory of lattice polytopes, in the sense that every lattice polytope $P$ (i.e., polytope with integer vertices) can be triangulated into empty simplices. 
In particular, it is very useful to have classifications or, at least, structural results, of all empty simplices in a given dimension. Classification will always be meant modulo unimodular equivalence: two lattice $d$-polytopes $P_1, P_2\subset \R^d$ are said  \emph{unimodularly equivalent} if there is an affine integer isomorphism between them; that is, a map $f:\R^d\to\R^d$ with integer coefficients and determinant $1$ such that $P_2=f(P_1)$
In dimension two every empty triangle is \emph{unimodular}; that is, its vertices form an affine basis for the lattice (this statement is equivalent to Pick's Theorem~\cite{BeckRobins}). In particular, all empty triangles are  \emph{unimodularly equivalent}. In dimension three this is 
no longer true. Yet, the classification of empty $3$-simplices is relatively simple:

\begin{theorem}[White 1964~\cite{White}]
\label{thm:white}
Every empty tetrahedron of determinant $q$ is unimodularly equivalent to 
\[
T(p,q):= \conv \{ (0,0,0), (1,0,0), (0,0,1), (p,q,1) \},
\]
for some $p\in\Z$ with $\gcd(p,q)=1$. Moreover, $T(p,q)$ is $\Z$-equivalent to $T(p',q)$ if and only if $p'=\pm p^{\pm 1} \pmod q$.
\end{theorem}

A key step in the proof of this theorem is the fact that all empty $3$-simplices have \emph{lattice width} equal to one according to the following definition: For each linear or affine functional $f:\R^d\to \R$ and any convex body $P\subset \R^d$ the width of $P$ with respect to $f$ is the difference between the maximum and minimum values of $f$ on $P$; that is, the length of the interval $f(P)$. The \emph{lattice width} of $P$ is the minimum width among all non-constant functionals with integer coefficients. In the case of empty $3$-simplices, once we now they have width one there is no loss of generality in assuming that two vertices lie in the plane $\{(x,y,z):z=0\}$ and another two in $\{(x,y,z):z=1\}$, which is essentially what White's classification says.

As sample applications of classifying  empty simplices, let us mention:
\begin{itemize}
\item It is a classical result of Knudsen et al.~\cite{Knudsen} that for every lattice polytope $P$ there is a constant $k\in \N$ such that the $k$-th dilation of $P$ has a unimodular triangulation; but it is open whether a uniform constant depending only on the dimension and not the particular polytope exists. In dimension three, heavily relying on the classification of empty simplices, 
Kantor and Sarkaria~\cite{KantorSarkaria} proved that $k=4$ does the job, and Santos and Ziegler extended this to any $k\in \N\setminus  \{41,2,3,5\}$. It is unknown whether a universal constant $k(d)$ valid for all polytopes of a given dimension $d$ exists, starting in dimension four.

\item In algebraic geometry empty simplices correspond to terminal quotient singularities. Their classification in dimension three is sometimes called the \emph{terminal lemma} (see some history and references in~\cite{MMM88}). 
\end{itemize}
In particular, quite some effort  has been done towards the classification of $4$-dimensional ones~\cite{BBBK09,BlancoHaaseHoffmanSantos,Bober09,Borisov-class,HZ00,MMM88,Sankaran90}.

So far the main known facts about empty $4$-simplices are:

\begin{theorem}[Properties of empty $4$-simplices,~\cite{BBBK09,BlancoHaaseHoffmanSantos,HZ00,MMM88}]
\label{thm:intro_previous_results}
\ 
\begin{enumerate}
\item There are infinitely many lattice empty $4$-simplices of width $1$ (e.g., cones over empty tetrahedra) and of width $2$~\cite{HZ00,MMM88}. 
\item Among the simplices of determinant $\le 1600$ there are $178$ of width three, one of width four, and none of larger width, and their maximum determinant is $179$. The computation up to determinant $1000$ was done by Haase and Ziegler~\cite{HZ00} and the extension up to $1600$ by Perriello~\cite{Perr08}.
\item The total amount of empty lattice $4$-simplices of width greater than 2 is finite~\cite{BlancoHaaseHoffmanSantos}. 
\item Every empty $4$-simplex is \emph{cyclic}~\cite{BBBK09}. The same is not true in dimension five. 
\end{enumerate}
\end{theorem}

%Of course the counting in parts (1) and (2) is modulo unimodular equivalence. 
The \emph{determinant} or \emph{normalized volume} of the $d$-simplex with vertices $v_0,\dots,v_d$ is
\[
\det\left(\begin{matrix} v_0&\dots&v_d \\ 1&\dots&1\end{matrix}\right).
\]
For a general polytope $P$, the normalized volume is just the Euclidean volume normalized to the lattice, so that unimodular simplices have volume one and every lattice polytope has integer volume. We distinguish between Euclidean and normalized volumes using $\vol(P)$ for the former and $\Vol(P)= d!\vol(P) $ for the latter. 
Section~\ref{sec:3hollow} is written mostly in terms of Euclidean volumes, since it relies in convex geometry results where that is the standard practice, but in Sections~\ref{sec:4hollow} and~\ref{sec:enumeration} we switch to $\Vol(P)$. For a (perhaps not empty) lattice simplex $P$, its determinant equals 
the order of the quotient group of the lattice $\Z^4$ by the sublattice generated by vertices of $P$. $P$ is called \emph{cyclic} if this quotient group is cyclic. 

Based on the exhaustive enumeration described in part (2) of Theorem~\ref{thm:intro_previous_results}, Haase and Ziegler~\cite{HZ00} conjectured that no further empty $4$-simplices of width larger than two existed. Our main result in this paper is the proof of that conjecture. This in particular provides an independent and more explicit proof of part (3) of Theorem~\ref{thm:intro_previous_results}:

\begin{theorem}
\label{thm:main}
The full list of empty $4$-simplices of width larger than two is as follows: there are 178 of width three, with determinants ranging from 
41 to 179, and a single one of width four, with determinant 101.
\end{theorem}

All empty $4$-simplices in this list (and presumably all $4$-dimensional empty simplices, although we do not have a proof of that) have a unimodular facet\footnote{Haase and Ziegler~\cite{HZ00} cite the diploma thesis~\cite{Wessels}  for a proof of this fact, but we have not been able to verify that proof. The closest result that we can certify is that, as a consequence of Lemma~\ref{lemma:coprime_facets}, in order for an empty $4$-simplex not to have any unimodular facet its determinant needs to be at least $2\cdot3\cdot5\cdot7\cdot11=2310$. We can also say that up to determinant $7600$ all empty $4$-simplices do have at least \emph{two} unimodular facets.}. This has the implication that they can be represented as
\[
\Delta(v):=\conv(e_1,e_2,e_3,e_4,v)
\]
for a certain $v\in \Z^4$. Table~\ref{table:full_list} gives the full list of empty $4$-simplices of width larger than two in this notation. Observe that the determinant of $\Delta(v)$ equals $\sum v_i-1$.

\begin{table}
\tiny
%\centerline{width 3}
\begin{tabular}{llll}
%!TEX root =Iglesias-santos-Arxiv-v4.tex

\begin{tabular}{lr}

\hline \footnotesize  $D=$ 41
& \footnotesize$( 4, 23, 25, -10)$ \\ %& \footnotesize$( 4, 23, 25, 31)$ \\ %& \footnotesize$( 4, 23, 25, 31)$ \\ %& \footnotesize$( 4, 23, 25, 31)$ \\ %& \footnotesize$( 4, 23, 25, 31)$ \\ %*

\hline \footnotesize $D=$ 43

& \footnotesize$( 3, 5, 11, 25)$ \\ %& \footnotesize$( 6, 25, 27, 29)$ \\ %& \footnotesize$( 8, 15, 26, 38)$ \\ %& \footnotesize$( 4, 23, 29, 31)$ \\ %& \footnotesize$( 3, 17, 31, 36)$ \\ %

& \footnotesize$( 4, 7, 9, 24)$ \\ %& \footnotesize$( 9, 11, 30, 37)$ \\ %& \footnotesize$(11, 15, 24, 37)$ \\ %& \footnotesize$( 4, 24, 26, 33)$ \\ %& \footnotesize$( 5, 7, 9, 23)$ \\ %

\hline \footnotesize $D=$ 44
& \footnotesize$( 3, 5, 13, 24)$ \\ %& \footnotesize$(13, 15, 25, 36)$ \\ %& \footnotesize$( 4, 9, 15, 17)$ \\ %& \footnotesize$( 3, 17, 32, 37)$ \\ %& \footnotesize$(0, 0, 0, 0, 0)$ \\ %

& \footnotesize$( 4, 7, 9, 25)$ \\ %& \footnotesize$(0, 0, 0, 0, 0)$ \\ %& \footnotesize$( 5, 9, 12, 19)$ \\ %& \footnotesize$( 5, 7, 9, 24)$ \\ %& \footnotesize$( 5, 19, 28, 37)$ \\ %

\hline \footnotesize $D=$ 47
& \footnotesize$( 3, 5, 13, 27)$ \\ %& \footnotesize$(14, 16, 27, 38)$ \\ %& \footnotesize$( 4, 19, 35, 37)$ \\ %& \footnotesize$( 7, 16, 29, 43)$ \\ %& \footnotesize$( 3, 7, 12, 26)$ \\ %

\hline \footnotesize $D=$ 48
& \footnotesize$( 3, 5, 13, 28)$ \\ %& \footnotesize$(0, 0, 0, 0, 0)$ \\ %& \footnotesize$( 4, 7, 9, 29)$ \\ %& \footnotesize$( 7, 20, 33, 37)$ \\ %& \footnotesize$( 7, 20, 33, 37)$ \\ %

\hline \footnotesize $D=$ 49
& \footnotesize$( 3, 5, 11, 31)$ \\ %& \footnotesize$( 6, 29, 31, 33)$ \\ %& \footnotesize$(10, 19, 33, 37)$ \\ %& \footnotesize$( 4, 9, 15, 22)$ \\ %& \footnotesize$( 3, 19, 36, 41)$ \\ %

& \footnotesize$( 4, 9, 11, 26)$ \\ %& \footnotesize$(10, 18, 34, 37)$ \\ %& \footnotesize$( 5, 8, 11, 26)$ \\ %& \footnotesize$( 9, 11, 13, 17)$ \\ %& \footnotesize$( 9, 17, 30, 43)$ \\ %

\hline \footnotesize $D=$ 50
& \footnotesize$( 4, 11, 17, 19)$ \\ %& \footnotesize$(0, 0, 0, 0, 0)$ \\ %& \footnotesize$( 3, 21, 36, 41)$ \\ %& \footnotesize$( 3, 17, 38, 43)$ \\ %& \footnotesize$( 7, 29, 31, 34)$ \\ %

\hline \footnotesize $D=$ 51
& \footnotesize$( 3, 5, 13, 31)$ \\ %& \footnotesize$(0, 0, 0, 0, 0)$ \\ %& \footnotesize$( 4, 28, 30, 41)$ \\ %& \footnotesize$( 4, 29, 31, 39)$ \\ %& \footnotesize$(13, 18, 28, 44)$ \\ %

& \footnotesize$( 3, 8, 13, 28)$ \\ %& \footnotesize$(0, 0, 0, 0, 0)$ \\ %& \footnotesize$( 6, 22, 32, 43)$ \\ %& \footnotesize$( 4, 19, 39, 41)$ \\ %& \footnotesize$( 5, 7, 9, 31)$ \\ %

\hline \footnotesize $D=$ 52

& \footnotesize$( 4, 9, 19, 21)$ \\ %& \footnotesize$(0, 0, 0, 0, 0)$ \\ %& \footnotesize$(15, 21, 29, 40)$ \\ %& \footnotesize$( 5, 8, 11, 29)$ \\ %& \footnotesize$( 5, 7, 9, 32)$ \\ %

\hline \footnotesize $D=$ 53

& \footnotesize$( 3, 5, 14, 32)$ \\ %& \footnotesize$( 7, 13, 16, 18)$ \\ %& \footnotesize$(10, 29, 32, 36)$ \\ %& \footnotesize$(11, 19, 28, 49)$ \\ %& \footnotesize$( 5, 28, 36, 38)$ \\ %

& \footnotesize$( 3, 5, 16, 30)$ \\ %& \footnotesize$(16, 18, 30, 43)$ \\ %& \footnotesize$(10, 18, 32, 47)$ \\ %& \footnotesize$( 3, 10, 18, 23)$ \\ %& \footnotesize$( 3, 23, 37, 44)$ \\ %

& \footnotesize$( 3, 8, 14, 29)$ \\ %& \footnotesize$( 8, 13, 15, 18)$ \\ %& \footnotesize$( 3, 20, 38, 46)$ \\ %& \footnotesize$( 7, 19, 32, 49)$ \\ %& \footnotesize$( 5, 11, 18, 20)$ \\ %

& \footnotesize$( 3, 19, 40, -8)$ \\ %& \footnotesize$(18, 22, 29, 38)$ \\ %& \footnotesize$( 6, 11, 14, 23)$ \\ %& \footnotesize$( 4, 30, 32, 41)$ \\ %& \footnotesize$( 5, 7, 9, 33)$ \\ %

& \footnotesize$( 3, 20, 39, -8)$ \\ %& \footnotesize$(11, 18, 38, 40)$ \\ %& \footnotesize$( 6, 8, 11, 29)$ \\ %& \footnotesize$( 4, 7, 9, 34)$ \\ %& \footnotesize$( 7, 29, 33, 38)$ \\ %

& \footnotesize$( 4, 9, 11, 30)$ \\ %& \footnotesize$(11, 19, 37, 40)$ \\ %& \footnotesize$( 6, 29, 32, 40)$ \\ %& \footnotesize$( 4, 29, 31, 43)$ \\ %& \footnotesize$( 5, 12, 14, 23)$ \\ %

\hline \footnotesize $D=$ 54

& \footnotesize$( 3, 5, 13, 34)$ \\ %& \footnotesize$(0, 0, 0, 0, 0)$ \\ %& \footnotesize$( 4, 11, 19, 21)$ \\ %& \footnotesize$(14, 25, 33, 37)$ \\ %& \footnotesize$(0, 0, 0, 0, 0)$ \\ %

\hline \footnotesize $D=$ 55

& \footnotesize$( 3, 5, 17, 31)$ \\ %& \footnotesize$( 8, 31, 35, 37)$ \\ %& \footnotesize$(0, 0, 0, 0, 0)$ \\ %& \footnotesize$(13, 16, 37, 45)$ \\ %& \footnotesize$( 3, 7, 16, 30)$ \\ %

& \footnotesize$( 3, 7, 12, 34)$ \\ %& \footnotesize$( 7, 16, 37, 51)$ \\ %& \footnotesize$( 3, 8, 14, 31)$ \\ %& \footnotesize$( 4, 23, 41, 43)$ \\ %& \footnotesize$( 8, 32, 34, 37)$ \\ %

 \hline \footnotesize $D=$ 56

& \footnotesize$( 3, 5, 17, 32)$ \\ %& \footnotesize$( 8, 13, 17, 19)$ \\ %& \footnotesize$(16, 19, 33, 45)$ \\ %& \footnotesize$( 3, 8, 13, 33)$ \\ %& \footnotesize$(0, 0, 0, 0, 0)$ \\ %

\hline \footnotesize $D=$ 57
& \footnotesize$( 3, 5, 13, 37)$ \\ %& \footnotesize$(0, 0, 0, 0, 0)$ \\ %& \footnotesize$( 4, 23, 43, 45)$ \\ %& \footnotesize$( 4, 22, 41, 48)$ \\ %& \footnotesize$( 3, 32, 37, 43)$ \\ %

\hline \footnotesize $D=$ 58
& \footnotesize$( 3, 5, 13, 38)$ \\ %& \footnotesize$(15, 26, 37, 39)$ \\ %& \footnotesize$( 4, 9, 11, 35)$ \\ %& \footnotesize$( 6, 9, 13, 31)$ \\ %& \footnotesize$(0, 0, 0, 0, 0)$ \\ %

& \footnotesize$(11, 32, 35, -19)$ \\ %& \footnotesize$( 7, 34, 37, 39)$ \\ %& \footnotesize$(0, 0, 0, 0, 0)$ \\ %& \footnotesize$( 3, 5, 14, 37)$ \\ %& \footnotesize$( 3, 11, 20, 25)$ \\ %, 

& \footnotesize$( 3, 31, 37, -12)$ \\ %& \footnotesize$( 4, 7, 9, 39)$ \\ %& \footnotesize$( 6, 13, 15, 25)$ \\ %& \footnotesize$( 7, 11, 16, 25)$ \\ %& \footnotesize$(0, 0, 0, 0, 0)$ \\ %

\hline \footnotesize $D=$ 59 %, 10, símplices

& \footnotesize$( 3, 5, 14, 38)$ \\ %& \footnotesize$( 7, 15, 18, 20)$ \\ %& \footnotesize$( 9, 12, 16, 23)$ \\ %& \footnotesize$( 4, 31, 38, 46)$ \\ %& \footnotesize$(14, 17, 40, 48)$ \\ %

& \footnotesize$( 3, 5, 19, 33)$ \\ %& \footnotesize$(18, 20, 33, 48)$ \\ %& \footnotesize$( 8, 12, 17, 23)$ \\ %& \footnotesize$(20, 28, 34, 37)$ \\ %& \footnotesize$( 3, 7, 16, 34)$ \\ %

& \footnotesize$( 3, 7, 12, 38)$ \\ %& \footnotesize$( 7, 20, 37, 55)$ \\ %& \footnotesize$( 3, 8, 17, 32)$ \\ %& \footnotesize$( 5, 24, 44, 46)$ \\ %& \footnotesize$( 9, 14, 17, 20)$ \\ %

& \footnotesize$( 3, 7, 13, 37)$ \\ %& \footnotesize$(20, 27, 35, 37)$ \\ %& \footnotesize$( 8, 15, 17, 20)$ \\ %& \footnotesize$( 4, 27, 38, 50)$ \\ %& \footnotesize$( 3, 8, 14, 35)$ \\ %

& \footnotesize$( 3, 11, 20, 26)$ \\ %& \footnotesize$(11, 13, 16, 20)$ \\ %& \footnotesize$( 3, 25, 43, 48)$ \\ %& \footnotesize$( 3, 26, 40, 50)$ \\ %& \footnotesize$(20, 25, 31, 43)$ \\ %

& \footnotesize$( 3, 32, 38, -13)$ \\ %& \footnotesize$( 7, 9, 20, 24)$ \\ %& \footnotesize$(17, 24, 32, 46)$ \\ %& \footnotesize$( 5, 14, 17, 24)$ \\ %& \footnotesize$( 7, 9, 12, 32)$ \\ %

& \footnotesize$( 4, 6, 15, 35)$ \\ %& \footnotesize$( 6, 11, 15, 28)$ \\ %& \footnotesize$( 4, 10, 19, 27)$ \\ %& \footnotesize$( 4, 35, 37, 43)$ \\ %& \footnotesize$( 8, 10, 15, 27)$ \\ %

& \footnotesize$( 4, 6, 17, 33)$ \\ %& \footnotesize$(15, 28, 36, 40)$ \\ %& \footnotesize$( 7, 10, 19, 24)$ \\ %& \footnotesize$( 5, 7, 17, 31)$ \\ %& \footnotesize$(12, 32, 34, 41)$ \\ %

& \footnotesize$( 4, 7, 15, 34)$ \\ %& \footnotesize$(11, 13, 15, 21)$ \\ %& \footnotesize$(12, 17, 40, 50)$ \\ %& \footnotesize$( 4, 31, 41, 43)$ \\ %& \footnotesize$( 5, 33, 36, 45)$ \\ %

& \footnotesize$( 4, 7, 17, 32)$ \\ %& \footnotesize$(13, 15, 40, 51)$ \\ %& \footnotesize$( 6, 17, 46, 50)$ \\ %& \footnotesize$( 7, 10, 12, 31)$ \\ %& \footnotesize$( 5, 9, 22, 24)$ \\ %

\hline \footnotesize $D=$ 60 %Mismo, símplice con, sus, 3, resultados

& \footnotesize$( 3, 8, 13, 37)$ \\ %& \footnotesize$(0, 0, 0, 0, 0)$ \\ %& \footnotesize$(0, 0, 0, 0, 0)$ \\ %& \footnotesize$( 4, 9, 11, 37)$ \\ %& \footnotesize$(11, 13, 16, 21)$ \\ %
%& \footnotesize$( 4, 9, 11, 37)$ \\ %& \footnotesize$(0, 0, 0, 0, 0)$ \\ %& \footnotesize$(0, 0, 0, 0, 0)$ \\ %& \footnotesize$(11, 13, 16, 21)$ \\ %& \footnotesize$( 3, 8, 13, 37)$ \\ %
%& \footnotesize$(11, 13, 16, 21)$ \\ %& \footnotesize$( 4, 9, 11, 37)$ \\ %& \footnotesize$(0, 0, 0, 0, 0)$ \\ %& \footnotesize$(0, 0, 0, 0, 0)$ \\ %& \footnotesize$(0, 0, 0, 0, 0)$ \\ %, 

\hline \footnotesize $D=$ 61

& \footnotesize$( 3, 5, 13, 41)$ \\ %& \footnotesize$(16, 27, 39, 41)$ \\ %& \footnotesize$( 4, 34, 36, 49)$ \\ %& \footnotesize$( 9, 25, 42, 47)$ \\ %& \footnotesize$( 3, 22, 46, 52)$ \\ %

& \footnotesize$( 3, 5, 17, 37)$ \\ %& \footnotesize$( 8, 35, 39, 41)$ \\ %& \footnotesize$(17, 21, 36, 49)$ \\ %& \footnotesize$( 5, 7, 18, 32)$ \\ %& \footnotesize$(18, 23, 33, 49)$ \\ %

& \footnotesize$( 3, 17, 49, -7)$ \\ %& \footnotesize$( 4, 35, 41, 43)$ \\ %& \footnotesize$( 4, 7, 18, 33)$ \\ %& \footnotesize$( 5, 35, 37, 46)$ \\ %& \footnotesize$( 7, 26, 44, 46)$ \\ %

& \footnotesize$( 3, 23, 45, -9)$ \\ %& \footnotesize$( 3, 33, 41, 46)$ \\ %& \footnotesize$( 6, 8, 11, 37)$ \\ %& \footnotesize$( 4, 19, 49, 51)$ \\ %& \footnotesize$( 5, 27, 41, 50)$ \\ %

& \footnotesize$( 3, 27, 41, -9)$ \\ %& \footnotesize$( 3, 27, 41, 52)$ \\ %& \footnotesize$( 3, 27, 41, 52)$ \\ %& \footnotesize$( 3, 27, 41, 52)$ \\ %& \footnotesize$( 3, 27, 41, 52)$ \\ %*

& \footnotesize$( 3, 34, 40, -15)$ \\ %& \footnotesize$( 5, 7, 9, 41)$ \\ %& \footnotesize$( 6, 9, 13, 34)$ \\ %& \footnotesize$( 8, 29, 35, 51)$ \\ %& \footnotesize$( 4, 23, 47, 49)$ \\ %

& \footnotesize$( 4, 33, 35, -10)$ \\ %& \footnotesize$( 7, 33, 37, 46)$ \\ %& \footnotesize$( 4, 35, 37, 47)$ \\ %& \footnotesize$( 7, 9, 13, 33)$ \\ %& \footnotesize$( 6, 34, 37, 46)$ \\ %

\hline \footnotesize $D=$ 62 %3, símplices

& \footnotesize$( 3, 5, 17, 38)$ \\ %& \footnotesize$( 8, 15, 19, 21)$ \\ %& \footnotesize$( 9, 25, 42, 49)$ \\ %& \footnotesize$( 7, 11, 16, 29)$ \\ %& \footnotesize$(0, 0, 0, 0, 0)$ \\ %

& \footnotesize$( 4, 7, 9, 43)$ \\ %& \footnotesize$(0, 0, 0, 0, 0)$ \\ %& \footnotesize$( 9, 26, 43, 47)$ \\ %& \footnotesize$( 7, 9, 13, 34)$ \\ %& \footnotesize$( 7, 10, 13, 33)$ \\ %

& \footnotesize$( 4, 11, 23, 25)$ \\ %& \footnotesize$(0, 0, 0, 0, 0)$ \\ %& \footnotesize$( 9, 17, 43, 56)$ \\ %& \footnotesize$( 7, 13, 16, 27)$ \\ %& \footnotesize$( 5, 7, 9, 42)$ \\ %

\hline \footnotesize $D=$ 63

& \footnotesize$( 3, 5, 19, 37)$ \\ %& \footnotesize$(0, 0, 0, 0, 0)$ \\ %& \footnotesize$(12, 34, 38, 43)$ \\ %& \footnotesize$( 8, 10, 13, 33)$ \\ %& \footnotesize$( 8, 22, 46, 51)$ \\ %

& \footnotesize$( 3, 8, 19, 34)$ \\ %& \footnotesize$(0, 0, 0, 0, 0)$ \\ %& \footnotesize$( 8, 37, 39, 43)$ \\ %& \footnotesize$(10, 33, 38, 46)$ \\ %& \footnotesize$( 5, 13, 22, 24)$ \\ %

\hline \footnotesize $D=$ 64

& \footnotesize$( 3, 5, 13, 44)$ \\ %& \footnotesize$(17, 28, 41, 43)$ \\ %& \footnotesize$( 4, 13, 23, 25)$ \\ %& \footnotesize$( 5, 36, 39, 49)$ \\ %& \footnotesize$(0, 0, 0, 0, 0)$ \\ %, 

& \footnotesize$(10, 13, 17, 25)$ \\ %& \footnotesize$(0, 0, 0, 0, 0)$ \\ %& \footnotesize$( 3, 5, 14, 43)$ \\ %& \footnotesize$( 3, 22, 49, 55)$ \\ %& \footnotesize$( 7, 38, 41, 43)$ \\ %

\hline \footnotesize $D=$ 65
& \footnotesize$( 3, 5, 22, 36)$ \\ %& \footnotesize$(20, 22, 36, 53)$ \\ %& \footnotesize$(0, 0, 0, 0, 0)$ \\ %& \footnotesize$( 3, 22, 50, 56)$ \\ %& \footnotesize$( 3, 27, 45, 56)$ \\ %

& \footnotesize$( 3, 8, 14, 41)$ \\ %& \footnotesize$( 8, 17, 19, 22)$ \\ %& \footnotesize$( 3, 24, 47, 57)$ \\ %& \footnotesize$(11, 14, 18, 23)$ \\ %& \footnotesize$( 6, 22, 46, 57)$ \\ %

& \footnotesize$( 3, 14, 22, 27)$ \\ %& \footnotesize$(17, 22, 36, 56)$ \\ %& \footnotesize$(12, 14, 17, 23)$ \\ %& \footnotesize$( 3, 23, 49, 56)$ \\ %& \footnotesize$( 4, 36, 38, 53)$ \\ %

& \footnotesize$( 3, 23, 48, -8)$ \\ %& \footnotesize$(14, 22, 46, 49)$ \\ %& \footnotesize$( 6, 14, 17, 29)$ \\ %& \footnotesize$( 4, 9, 11, 42)$ \\ %& \footnotesize$( 6, 8, 11, 41)$ \\ %

\hline \footnotesize $D=$ 67
& \footnotesize$( 3, 5, 13, 47)$ \\ %& \footnotesize$(18, 29, 43, 45)$ \\ %& \footnotesize$( 4, 27, 51, 53)$ \\ %& \footnotesize$(17, 31, 41, 46)$ \\ %& \footnotesize$( 4, 10, 17, 37)$ \\ %

& \footnotesize$( 3, 5, 17, 43)$ \\ %& \footnotesize$( 8, 39, 43, 45)$ \\ %& \footnotesize$(10, 27, 45, 53)$ \\ %& \footnotesize$( 4, 29, 47, 55)$ \\ %& \footnotesize$( 3, 37, 42, 53)$ \\ %

& \footnotesize$( 3, 5, 23, 37)$ \\ %& \footnotesize$(10, 37, 43, 45)$ \\ %& \footnotesize$( 6, 27, 49, 53)$ \\ %& \footnotesize$(26, 29, 35, 45)$ \\ %& \footnotesize$( 3, 29, 47, 56)$ \\ %

& \footnotesize$( 3, 17, 55, -7)$ \\ %& \footnotesize$( 4, 39, 45, 47)$ \\ %& \footnotesize$( 4, 28, 48, 55)$ \\ %& \footnotesize$( 5, 7, 17, 39)$ \\ %& \footnotesize$(10, 12, 19, 27)$ \\ %

\end{tabular}& 
\begin{tabular}{lr}

& \footnotesize$( 4, 7, 9, 48)$ \\ %& \footnotesize$(15, 17, 48, 55)$ \\ %& \footnotesize$( 9, 37, 41, 48)$ \\ %& \footnotesize$( 7, 15, 17, 29)$ \\ %& \footnotesize$( 4, 7, 18, 39)$ \\ %

& \footnotesize$( 4, 7, 15, 42)$ \\ %& \footnotesize$(13, 15, 17, 23)$ \\ %& \footnotesize$( 9, 17, 48, 61)$ \\ %& \footnotesize$( 4, 9, 24, 31)$ \\ %& \footnotesize$( 8, 11, 14, 35)$ \\ %

\hline \footnotesize $D=$ 68
& \footnotesize$( 3, 5, 13, 48)$ \\ %& \footnotesize$(21, 23, 41, 52)$ \\ %& \footnotesize$( 4, 11, 13, 41)$ \\ %& \footnotesize$( 5, 12, 21, 31)$ \\ %& \footnotesize$(0, 0, 0, 0, 0)$ \\ %

& \footnotesize$( 3, 5, 23, 38)$ \\ %& \footnotesize$(10, 15, 21, 23)$ \\ %& \footnotesize$( 6, 9, 13, 41)$ \\ %& \footnotesize$( 3, 22, 53, 59)$ \\ %& \footnotesize$(0, 0, 0, 0, 0)$ \\ %

\hline \footnotesize $D=$ 69

& \footnotesize$( 3, 5, 19, 43)$ \\ %& \footnotesize$(0, 0, 0, 0, 0)$ \\ %& \footnotesize$(10, 14, 19, 27)$ \\ %& \footnotesize$( 5, 7, 18, 40)$ \\ %& \footnotesize$(14, 24, 40, 61)$ \\ %

& \footnotesize$( 3, 8, 22, 37)$ \\ %& \footnotesize$(0, 0, 0, 0, 0)$ \\ %& \footnotesize$( 4, 26, 49, 60)$ \\ %& \footnotesize$( 3, 14, 22, 31)$ \\ %& \footnotesize$( 5, 28, 52, 54)$ \\ %

\hline \footnotesize $D=$ 71  %9, símplices

& \footnotesize$( 3, 5, 13, 51)$ \\ %& \footnotesize$(22, 24, 43, 54)$ \\ %& \footnotesize$( 4, 40, 42, 57)$ \\ %& \footnotesize$( 7, 11, 16, 38)$ \\ %& \footnotesize$(18, 25, 39, 61)$ \\ %

& \footnotesize$( 3, 5, 22, 42)$ \\ %& \footnotesize$(22, 24, 40, 57)$ \\ %& \footnotesize$(20, 24, 42, 57)$ \\ %& \footnotesize$( 3, 11, 16, 42)$ \\ %& \footnotesize$( 5, 13, 22, 32)$ \\ %

& \footnotesize$( 3, 8, 19, 42)$ \\ %& \footnotesize$(21, 24, 41, 57)$ \\ %& \footnotesize$( 9, 42, 44, 48)$ \\ %& \footnotesize$( 9, 15, 22, 26)$ \\ %& \footnotesize$( 5, 8, 22, 37)$ \\ %

& \footnotesize$( 3, 8, 23, 38)$ \\ %& \footnotesize$(11, 16, 21, 24)$ \\ %& \footnotesize$( 6, 9, 13, 44)$ \\ %& \footnotesize$(12, 34, 40, 57)$ \\ %& \footnotesize$( 5, 11, 13, 43)$ \\ %

& \footnotesize$( 3, 11, 20, 38)$ \\ %& \footnotesize$(11, 17, 20, 24)$ \\ %& \footnotesize$( 3, 13, 24, 32)$ \\ %& \footnotesize$( 3, 32, 46, 62)$ \\ %& \footnotesize$(13, 24, 43, 63)$ \\ %

& \footnotesize$( 3, 13, 18, 38)$ \\ %& \footnotesize$(11, 24, 43, 65)$ \\ %& \footnotesize$( 8, 11, 15, 38)$ \\ %& \footnotesize$( 4, 19, 59, 61)$ \\ %& \footnotesize$( 7, 9, 13, 43)$ \\ %

& \footnotesize$(4, 7, 18, 43)$ \\ %& \footnotesize$( 7, 16, 18, 31)$ \\ %& \footnotesize$( 4, 38, 40, 61)$ \\ %& \footnotesize$( 4, 41, 43, 55)$ \\ %& \footnotesize$(18, 26, 38, 61)$ \\ %

& \footnotesize$(4, 11, 18, 39)$ \\ %& \footnotesize$( 8, 15, 18, 31)$ \\ %& \footnotesize$(13, 19, 50, 61)$ \\ %& \footnotesize$( 4, 27, 55, 57)$ \\ %& \footnotesize$( 5, 7, 9, 51)$ \\ %

& \footnotesize$( 4, 39, 49, -20)$ \\ %& \footnotesize$(5, 8, 18, 41)$ \\ %& \footnotesize$( 9, 26, 51, 57)$ \\ %& \footnotesize$( 5, 12, 26, 29)$ \\ %& \footnotesize$( 6, 39, 41, 57)$ \\ %

\hline \footnotesize $D=$ 73 %5, símplices

& \footnotesize$( 3, 5, 17, 49)$ \\ %& \footnotesize$( 8, 43, 47, 49)$ \\ %& \footnotesize$(14, 34, 44, 55)$ \\ %& \footnotesize$( 4, 10, 17, 43)$ \\ %& \footnotesize$( 3, 22, 58, 64)$ \\ %

& \footnotesize$( 3, 8, 14, 49)$ \\ %& \footnotesize$( 8, 44, 46, 49)$ \\ %& \footnotesize$( 3, 27, 53, 64)$ \\ %& \footnotesize$( 5, 33, 47, 62)$ \\ %& \footnotesize$( 3, 31, 49, 64)$ \\ %

& \footnotesize$( 3, 11, 17, 43)$ \\ %& \footnotesize$(10, 43, 45, 49)$ \\ %& \footnotesize$(13, 16, 20, 25)$ \\ %& \footnotesize$(17, 38, 43, 49)$ \\ %& \footnotesize$( 3, 17, 22, 32)$ \\ %

& \footnotesize$( 3, 26, 54, -9)$ \\ %& \footnotesize$( 3, 40, 49, 55)$ \\ %& \footnotesize$(20, 26, 42, 59)$ \\ %& \footnotesize$( 4, 23, 59, 61)$ \\ %& \footnotesize$( 6, 8, 11, 49)$ \\ %

& \footnotesize$( 4, 39, 47, -16)$ \\ %& \footnotesize$( 4, 43, 45, 55)$ \\ %& \footnotesize$(13, 15, 21, 25)$ \\ %& \footnotesize$( 5, 14, 17, 38)$ \\ %& \footnotesize$( 7, 41, 44, 55)$ \\ %

\hline \footnotesize $D=$ 74 %3, símplices

& \footnotesize$( 3, 5, 13, 54)$ \\ %& \footnotesize$(23, 25, 45, 56)$ \\ %& \footnotesize$( 4, 15, 27, 29)$ \\ %& \footnotesize$(11, 30, 51, 57)$ \\ %& \footnotesize$(0, 0, 0, 0, 0)$ \\ %

& \footnotesize$( 7, 20, 23, 25)$ \\ %& \footnotesize$( 7, 39, 50, 53)$ \\ %& \footnotesize$(0, 0, 0, 0, 0)$ \\ %& \footnotesize$(12, 15, 19, 29)$ \\ %& \footnotesize$(3, 5, 14, 53)$ \\ %

& \footnotesize$( 3, 5, 23, 44)$ \\ %& \footnotesize$(10, 17, 23, 25)$ \\ %& \footnotesize$( 6, 15, 25, 29)$ \\ %& \footnotesize$( 3, 29, 56, 61)$ \\ %& \footnotesize$(0, 0, 0, 0, 0)$ \\ %

\hline \footnotesize $D=$ 75
& \footnotesize$(3, 8, 19, 46)$ \\ %& \footnotesize$(0, 0, 0, 0, 0)$ \\ %& \footnotesize$( 7, 9, 13, 47)$ \\ %& \footnotesize$( 4, 41, 43, 63)$ \\ %& \footnotesize$(11, 31, 52, 57)$ \\ %

& \footnotesize$(3, 41, 46, -14)$ \\ %& \footnotesize$(0, 0, 0, 0, 0)$ \\ %& \footnotesize$( 4, 11, 19, 42)$ \\ %& \footnotesize$( 4, 31, 57, 59)$ \\ %& \footnotesize$(14, 16, 19, 27)$ \\ %

& \footnotesize$(4, 7, 19, 46)$ \\ %& \footnotesize$(14, 17, 19, 26)$ \\ %& \footnotesize$( 8, 43, 47, 53)$ \\ %& \footnotesize$( 4, 41, 47, 59)$ \\ %& \footnotesize$( 8, 11, 26, 31)$ \\ %

& \footnotesize$(6, 8, 19, 43)$ \\ %& \footnotesize$(0, 0, 0, 0, 0)$ \\ %& \footnotesize$(4, 7, 18, 47)$ \\ %& \footnotesize$( 4, 43, 51, 53)$ \\ %& \footnotesize$( 7, 17, 19, 33)$ \\ %

\hline \footnotesize $D=$ 76 %2, símplices

& \footnotesize$(4, 9, 11, 53)$ \\ %& \footnotesize$(0, 0, 0, 0, 0)$ \\ %& \footnotesize$(8, 11, 17, 41)$ \\ %& \footnotesize$(7, 9, 13, 48)$ \\ %& \footnotesize$(7, 17, 20, 33)$ \\ %

& \footnotesize$(5, 14, 17, 41)$ \\ %& \footnotesize$(7, 27, 58, 61)$ \\ %& \footnotesize$(0, 0, 0, 0, 0)$ \\ %& \footnotesize$(9, 11, 26, 31)$ \\ %& \footnotesize$(7, 11, 13, 46)$ \\ %

\hline \footnotesize $D=$ 77 %5, símplices

& \footnotesize$( 3, 5, 13, 57)$ \\ %& \footnotesize$(24, 26, 47, 58)$ \\ %& \footnotesize$( 4, 31, 59, 61)$ \\ %& \footnotesize$( 6, 43, 47, 59)$ \\ %& \footnotesize$( 4, 43, 50, 58)$ \\ %

& \footnotesize$( 3, 8, 26, 41)$ \\ %& \footnotesize$(12, 17, 23, 26)$ \\ %& \footnotesize$(16, 29, 43, 67)$ \\ %& \footnotesize$( 3, 31, 53, 68)$ \\ %& \footnotesize$( 5, 43, 45, 62)$ \\ %

& \footnotesize$( 3, 17, 26, 32)$ \\ %& \footnotesize$(15, 17, 20, 26)$ \\ %& \footnotesize$( 3, 27, 57, 68)$ \\ %& \footnotesize$( 3, 26, 58, 68)$ \\ %& \footnotesize$( 4, 36, 50, 65)$ \\ %

& \footnotesize$( 3, 26, 59, -10)$ \\ %& \footnotesize$( 6, 17, 26, 29)$ \\ %& \footnotesize$( 3, 30, 54, 68)$ \\ %& \footnotesize$( 8, 10, 13, 47)$ \\ %& \footnotesize$( 8, 18, 23, 29)$ \\ %

& \footnotesize$( 8, 21, 23, 26)$ \\ %& \footnotesize$( 7, 16, 26, 29)$ \\ %& \footnotesize$(0, 0, 0, 0, 0)$ \\ %& \footnotesize$( 3, 29, 56, 67)$ \\ %& \footnotesize$(3, 8, 14, 53)$ \\ %, , 

\hline \footnotesize $D=$ 79 %12, símplices

& \footnotesize$( 3, 5, 14, 58)$ \\ %& \footnotesize$( 7, 48, 51, 53)$ \\ %& \footnotesize$(13, 16, 20, 31)$ \\ %& \footnotesize$(17, 28, 41, 73)$ \\ %& \footnotesize$( 4, 15, 27, 34)$ \\ %

& \footnotesize$( 3, 5, 17, 55)$ \\ %& \footnotesize$( 8, 47, 51, 53)$ \\ %& \footnotesize$(16, 31, 44, 68)$ \\ %& \footnotesize$( 9, 14, 20, 37)$ \\ %& \footnotesize$( 4, 10, 23, 43)$ \\ %

& \footnotesize$( 3, 5, 19, 53)$ \\ %& \footnotesize$(20, 35, 51, 53)$ \\ %& \footnotesize$(12, 16, 21, 31)$ \\ %& \footnotesize$( 4, 18, 25, 33)$ \\ %& \footnotesize$( 3, 22, 64, 70)$ \\ %

& \footnotesize$( 3, 5, 29, 43)$ \\ %& \footnotesize$(12, 43, 51, 53)$ \\ %& \footnotesize$(10, 16, 23, 31)$ \\ %& \footnotesize$( 8, 30, 53, 68)$ \\ %& \footnotesize$( 3, 33, 55, 68)$ \\ %

& \footnotesize$( 3, 8, 14, 55)$ \\ %& \footnotesize$( 8, 48, 50, 53)$ \\ %& \footnotesize$( 3, 10, 18, 49)$ \\ %& \footnotesize$(13, 17, 22, 28)$ \\ %& \footnotesize$(10, 23, 53, 73)$ \\ %

& \footnotesize$( 3, 8, 19, 50)$ \\ %& \footnotesize$(20, 36, 50, 53)$ \\ %& \footnotesize$(10, 47, 49, 53)$ \\ %& \footnotesize$( 4, 14, 25, 37)$ \\ %& \footnotesize$( 3, 11, 17, 49)$ \\ %

& \footnotesize$( 3, 11, 20, 46)$ \\ %& \footnotesize$(11, 46, 49, 53)$ \\ %& \footnotesize$( 3, 36, 50, 70)$ \\ %& \footnotesize$( 4, 35, 53, 67)$ \\ %& \footnotesize$( 3, 36, 53, 67)$ \\ %

& \footnotesize$( 3, 19, 25, 33)$ \\ %& \footnotesize$(18, 20, 53, 68)$ \\ %& \footnotesize$(4, 7, 25, 44)$ \\ %& \footnotesize$( 5, 19, 22, 34)$ \\ %& \footnotesize$( 9, 12, 16, 43)$ \\ %

& \footnotesize$( 3, 36, 56, -15)$ \\ %& \footnotesize$( 5, 34, 53, 67)$ \\ %& \footnotesize$( 7, 11, 16, 46)$ \\ %& \footnotesize$( 5, 7, 24, 44)$ \\ %& \footnotesize$( 9, 16, 21, 34)$ \\ %

& \footnotesize$( 4, 15, 17, 44)$ \\ %& \footnotesize$(16, 20, 55, 68)$ \\ %& \footnotesize$( 5, 41, 55, 58)$ \\ %& \footnotesize$(14, 16, 23, 27)$ \\ %& \footnotesize$( 5, 9, 23, 43)$ \\ %

& \footnotesize$( 5, 8, 23, 44)$ \\ %& \footnotesize$( 7, 16, 27, 30)$ \\ %& \footnotesize$( 7, 10, 29, 34)$ \\ %& \footnotesize$(29, 34, 41, 55)$ \\ %& \footnotesize$( 7, 9, 30, 34)$ \\ %, 

& \footnotesize$( 5, 9, 32, 34)$ \\ %& \footnotesize$( 9, 14, 16, 41)$ \\ %& \footnotesize$( 5, 14, 17, 44)$ \\ %& \footnotesize$(17, 27, 42, 73)$ \\ %& \footnotesize$( 7, 13, 16, 44)$ \\ %

\hline \footnotesize $D=$ 81 %1, símplice

& \footnotesize$( 3, 5, 13, 61)$ \\ %& \footnotesize$(0, 0, 0, 0)$ \\ %& \footnotesize$( 4, 46, 48, 65)$ \\ %& \footnotesize$( 6, 14, 25, 37)$ \\ %& \footnotesize$( 4, 29, 61, 69)$ \\ %

\hline \footnotesize $D=$ 82 %1, símplice 

& \footnotesize$( 4, 19, 29, 31)$ \\ %& \footnotesize$(0, 0, 0, 0)$ \\ %& \footnotesize$( 7, 13, 30, 33)$ \\ %& \footnotesize$( 5, 14, 17, 47)$ \\ %& \footnotesize$( 7, 45, 47, 66)$ \\ %

\hline \footnotesize $D=$ 83 %10, símplices, 

& \footnotesize$(3, 5, 26, 50)$ \\ %& \footnotesize$(11, 19, 26, 28)$ \\ %& \footnotesize$(16, 28, 50, 73)$ \\ %& \footnotesize$(3, 16, 30, 35)$ \\ %& \footnotesize$(5, 36, 58, 68)$ \\ %

& \footnotesize$(3, 5, 28, 48)$ \\ %& \footnotesize$(26, 28, 46, 67)$ \\ %& \footnotesize$(7, 11, 16, 50)$ \\ %& \footnotesize$(3, 22, 68, 74)$ \\ %& \footnotesize$(12, 34, 57, 64)$ \\ %

& \footnotesize$(3, 8, 14, 59)$ \\ %& \footnotesize$(8, 23, 25, 28)$ \\ %& \footnotesize$(3, 10, 19, 52)$ \\ %& \footnotesize$(6, 35, 61, 65)$ \\ %& \footnotesize$(28, 38, 49, 52)$ \\ %

& \footnotesize$( 3, 11, 20, 50)$ \\ %& \footnotesize$(11, 21, 24, 28)$ \\ %& \footnotesize$(3, 45, 51, 68)$ \\ %& \footnotesize$(4, 39, 54, 70)$ \\ %& \footnotesize$(5, 28, 66, 68)$ \\ %

& \footnotesize$( 3, 14, 20, 47)$ \\ %& \footnotesize$(12, 21, 23, 28)$ \\ %& \footnotesize$(6, 46, 50, 65)$ \\ %& \footnotesize$(4, 35, 54, 74)$ \\ %& \footnotesize$(5, 7, 19, 53)$ \\ %

& \footnotesize$( 3, 14, 23, 44)$ \\ %& \footnotesize$(13, 20, 23, 28)$ \\ %& \footnotesize$(6, 28, 65, 68)$ \\ %& \footnotesize$( 3, 45, 54, 65)$ \\ %& \footnotesize$(11, 17, 24, 32)$ \\ %

& \footnotesize$( 3, 20, 26, 35)$ \\ %& \footnotesize$(16, 19, 21, 28)$ \\ %& \footnotesize$(4, 7, 19, 54)$ \\ %& \footnotesize$(12, 16, 21, 35)$ \\ %& \footnotesize$( 4, 19, 26, 35)$ \\ %

& \footnotesize$( 3, 37, 59, -15)$ \\ %& \footnotesize$(5, 8, 28, 43)$ \\ %& \footnotesize$( 9, 50, 52, 56)$ \\ %& \footnotesize$( 5, 38, 52, 72)$ \\ %& \footnotesize$( 8, 11, 15, 50)$ \\ %

& \footnotesize$( 3, 44, 52, -15)$ \\ %& \footnotesize$( 5, 13, 28, 38)$ \\ %& \footnotesize$( 6, 17, 29, 32)$ \\ %& \footnotesize$( 8, 37, 59, 63)$ \\ %& \footnotesize$( 9, 11, 14, 50)$ \\ %

& \footnotesize$( 5, 12, 14, 53)$ \\ %& \footnotesize$( 6, 47, 50, 64)$ \\ %& \footnotesize$( 7, 44, 48, 68)$ \\ %& \footnotesize$( 6, 11, 14, 53)$ \\ %& \footnotesize$( 6, 14, 17, 47)$ \\ %

\hline \footnotesize $D=$ 84 %%1, símplice 

& \footnotesize$(3, 5, 13, 64)$ \\ %& \footnotesize$(0, 0, 0, 0, 0)$ \\ %& \footnotesize$(4, 17, 31, 33)$ \\ %& \footnotesize$(8, 13, 19, 45)$ \\ %& \footnotesize$(0, 0, 0, 0, 0)$ \\ %

\hline \footnotesize $D=$ 85 %4, símplices,, 
& \footnotesize$( 3, 5, 32, 46)$ \\ %& \footnotesize$(13, 46, 55, 57)$ \\ %& \footnotesize$(0, 0, 0, 0, 0)$ \\ %& \footnotesize$( 8, 45, 57, 61)$ \\ %& \footnotesize$( 3, 35, 61, 72)$ \\ %

\end{tabular}& 
\begin{tabular}{lr}

& \footnotesize$( 3, 31, 64, -12)$ \\ %& \footnotesize$( 4, 7, 18, 57)$ \\ %& \footnotesize$(11, 47, 52, 61)$ \\ %& \footnotesize$( 4, 46, 57, 64)$ \\ %& \footnotesize$( 7, 38, 62, 64)$ \\ %

& \footnotesize$( 3, 32, 63, -12)$ \\ %& \footnotesize$( 4, 46, 48, 73)$ \\ %& \footnotesize$( 6, 8, 11, 61)$ \\ %& \footnotesize$( 4, 27, 69, 71)$ \\ %& \footnotesize$( 7, 31, 64, 69)$ \\ %

& \footnotesize$( 3, 46, 58, -21)$ \\ %& \footnotesize$( 7, 9, 13, 57)$ \\ %& \footnotesize$( 6, 32, 61, 72)$ \\ %& \footnotesize$( 8, 19, 22, 37)$ \\ %& \footnotesize$( 4, 23, 71, 73)$ \\ %

\hline \footnotesize $D=$ 87 %2, símplices

& \footnotesize$( 3, 5, 13, 67)$ \\ %& \footnotesize$(0, 0, 0, 0, 0)$ \\ %& \footnotesize$( 4, 35, 67, 69)$ \\ %& \footnotesize$(13, 35, 60, 67)$ \\ %& \footnotesize$( 5, 13, 22, 48)$ \\ %

& \footnotesize$( 3, 8, 22, 55)$ \\ %& \footnotesize$(0, 0, 0, 0, 0)$ \\ %& \footnotesize$( 4, 11, 19, 54)$ \\ %& \footnotesize$( 4, 41, 55, 75)$ \\ %& \footnotesize$(17, 19, 22, 30)$ \\ %

\hline \footnotesize $D=$ 89 %3, símplices, 

& \footnotesize$( 3, 5, 34, 48)$ \\ %& \footnotesize$(28, 30, 48, 73)$ \\ %& \footnotesize$(11, 18, 26, 35)$ \\ %& \footnotesize$(13, 30, 55, 81)$ \\ %& \footnotesize$( 3, 13, 24, 50)$ \\ %

& \footnotesize$( 4, 7, 18, 61)$ \\ %& \footnotesize$( 7, 40, 65, 67)$ \\ %& \footnotesize$( 4, 51, 61, 63)$ \\ %& \footnotesize$( 5, 51, 54, 69)$ \\ %& \footnotesize$( 7, 51, 54, 67)$ \\ %

& \footnotesize$( 4, 49, 51, -14)$ \\ %& \footnotesize$(10, 48, 54, 67)$ \\ %& \footnotesize$( 9, 13, 20, 48)$ \\ %& \footnotesize$( 7, 9, 13, 61)$ \\ %& \footnotesize$(10, 13, 19, 48)$ \\ %

\hline \footnotesize $D=$ 91 %2, símplices
& \footnotesize$( 3, 5, 23, 61)$ \\ %& \footnotesize$(10, 53, 59, 61)$ \\ %& \footnotesize$( 6, 50, 54, 73)$ \\ %& \footnotesize$( 4, 29, 71, 79)$ \\ %& \footnotesize$( 3, 22, 76, 82)$ \\ %

& \footnotesize$( 3, 8, 23, 58)$ \\ %& \footnotesize$(11, 53, 58, 61)$ \\ %& \footnotesize$(11, 54, 57, 61)$ \\ %& \footnotesize$( 4, 41, 59, 79)$ \\ %& \footnotesize$( 3, 11, 20, 58)$ \\ %

\hline \footnotesize $D=$ 94 %%1, símplice
& \footnotesize$( 3, 5, 13, 74)$ \\ %& \footnotesize$(27, 38, 61, 63)$ \\ %& \footnotesize$( 4, 19, 35, 37)$ \\ %& \footnotesize$( 7, 16, 29, 43)$ \\ %& \footnotesize$(0, 0, 0, 0, )$ \\ %

\hline \footnotesize $D=$ 95

& \footnotesize$(3, 5, 32, 56)$ \\ %& \footnotesize$(13, 21, 30, 32)$ \\ %& \footnotesize$(0, 0, 0, 0, 0)$ \\ %& \footnotesize$( 3, 22, 80, 86)$ \\ %& \footnotesize$( 5, 13, 22, 56)$ \\ %

& \footnotesize$(3, 5, 37, 51)$ \\ %& \footnotesize$(30, 32, 51, 78)$ \\ %& \footnotesize$(0, 0, 0, 0, 0)$ \\ %& \footnotesize$( 5, 18, 32, 41)$ \\ %& \footnotesize$( 3, 41, 67, 80)$ \\ %

& \footnotesize$(3, 8, 23, 62)$ \\ %& \footnotesize$(11, 24, 29, 32)$ \\ %& \footnotesize$( 9, 12, 16, 59)$ \\ %& \footnotesize$( 4, 51, 62, 74)$ \\ %& \footnotesize$(6, 23, 26, 41)$ \\ %

& \footnotesize$(3, 8, 26, 59)$ \\ %& \footnotesize$(12, 23, 29, 32)$ \\ %& \footnotesize$(12, 52, 59, 68)$ \\ %& \footnotesize$( 7, 11, 16, 62)$ \\ %& \footnotesize$(6, 8, 29, 53)$ \\ %

& \footnotesize$(3, 11, 28, 54)$ \\ %& \footnotesize$(28, 32, 54, 77)$ \\ %& \footnotesize$(17, 21, 26, 32)$ \\ %& \footnotesize$( 3, 17, 32, 44)$ \\ %& \footnotesize$( 3, 44, 58, 86)$ \\ %

& \footnotesize$(3, 17, 26, 50)$ \\ %& \footnotesize$(15, 23, 26, 32)$ \\ %& \footnotesize$(11, 25, 28, 32)$ \\ %& \footnotesize$( 3, 11, 20, 62)$ \\ %& \footnotesize$(0, 0, 0, 0, 0)$ \\ %

\hline \footnotesize $D=$ 97 %, 6, símplices

& \footnotesize$( 4, 53, 65, -24)$ \\ %& \footnotesize$(6, 8, 11, 73)$ \\ %& \footnotesize$(11, 53, 61, 70)$ \\ %& \footnotesize$( 3, 35, 72, 85)$ \\ %& \footnotesize$( 4, 31, 79, 81)$ \\ %

& \footnotesize$( 5, 54, 63, -24)$ \\ %& \footnotesize$(28, 39, 63, 65)$ \\ %& \footnotesize$( 9, 15, 22, 52)$ \\ %& \footnotesize$(3, 5, 13, 77)$ \\ %& \footnotesize$( 4, 39, 75, 77)$ \\ %

& \footnotesize$( 6, 39, 73, -20)$ \\ %& \footnotesize$( 4, 42, 68, 81)$ \\ %& \footnotesize$(3, 5, 23, 67)$ \\ %& \footnotesize$( 4, 38, 73, 80)$ \\ %& \footnotesize$(10, 57, 63, 65)$ \\ %

& \footnotesize$( 7, 17, 20, 54)$ \\ %& \footnotesize$(11, 14, 20, 53)$ \\ %& \footnotesize$(11, 40, 71, 73)$ \\ %& \footnotesize$(4, 7, 34, 53)$ \\ %& \footnotesize$( 9, 14, 34, 41)$ \\ %

& \footnotesize$( 7, 41, 71, -21)$ \\ %& \footnotesize$(3, 8, 14, 73)$ \\ %& \footnotesize$( 3, 36, 71, 85)$ \\ %& \footnotesize$( 4, 41, 65, 85)$ \\ %& \footnotesize$( 8, 60, 62, 65)$ \\ %

& \footnotesize$( 9, 39, 70, -20)$ \\ %& \footnotesize$( 3, 13, 28, 54)$ \\ %& \footnotesize$(3, 5, 38, 52)$ \\ %& \footnotesize$(23, 28, 65, 79)$ \\ %& \footnotesize$(15, 52, 63, 65)$ \\ %

\hline \footnotesize $D=$ 101 %4, símplices

& \footnotesize$( 3, 5, 32, 62)$ \\ %& \footnotesize$(13, 23, 32, 34)$ \\ %& \footnotesize$(28, 34, 60, 81)$ \\ %& \footnotesize$( 3, 17, 22, 60)$ \\ %& \footnotesize$( 6, 44, 70, 83)$ \\ %

& \footnotesize$( 3, 8, 14, 77)$ \\ %& \footnotesize$( 8, 29, 31, 34)$ \\ %& \footnotesize$( 3, 38, 74, 88)$ \\ %& \footnotesize$( 7, 45, 65, 86)$ \\ %& \footnotesize$( 9, 21, 34, 38)$ \\ %

& \footnotesize$( 3, 8, 38, 53)$ \\ %& \footnotesize$(16, 21, 31, 34)$ \\ %& \footnotesize$( 6, 38, 71, 88)$ \\ %& \footnotesize$( 8, 37, 77, 81)$ \\ %& \footnotesize$( 5, 17, 19, 61)$ \\ %

& \footnotesize${\bf( 6, 14, 17, 65)}$ \\ %& \footnotesize$( 6, 14, 17, 65)$ \\ %& \footnotesize$( 6, 14, 17, 65)$ \\ %& \footnotesize$( 6, 14, 17, 65)$ \\ %& \footnotesize$( 6, 14, 17, 65)$ \\ %*

\hline \footnotesize $D=$ 103 %7, símplices
& \footnotesize$(3, 5, 29, 67)$ \\ %& \footnotesize$(12, 59, 67, 69)$ \\ %& \footnotesize$(20, 56, 62, 69)$ \\ %& \footnotesize$( 7, 19, 32, 46)$ \\ %& \footnotesize$( 3, 20, 38, 43)$ \\ %

& \footnotesize$(3, 5, 38, 58)$ \\ %& \footnotesize$(15, 56, 67, 69)$ \\ %& \footnotesize$( 9, 13, 20, 62)$ \\ %& \footnotesize$( 8, 19, 31, 46)$ \\ %& \footnotesize$(10, 16, 23, 55)$ \\ %

& \footnotesize$(3, 20, 26, 55)$ \\ %& \footnotesize$(16, 60, 62, 69)$ \\ %& \footnotesize$(5, 9, 23, 67)$ \\ %& \footnotesize$( 4, 23, 89, 91)$ \\ %& \footnotesize$( 9, 15, 22, 58)$ \\ %

& \footnotesize$(4, 7, 18, 75)$ \\ %& \footnotesize$( 7, 24, 26, 47)$ \\ %& \footnotesize$( 4, 59, 71, 73)$ \\ %& \footnotesize$(13, 57, 63, 74)$ \\ %& \footnotesize$( 8, 11, 26, 59)$ \\ %

& \footnotesize$(4, 11, 26, 63)$ \\ %& \footnotesize$(10, 23, 26, 45)$ \\ %& \footnotesize$(7, 9, 13, 75)$ \\ %& \footnotesize$( 4, 57, 59, 87)$ \\ %& \footnotesize$( 8, 18, 31, 47)$ \\ %

& \footnotesize$(4, 55, 73, -28)$ \\ %& \footnotesize$( 7, 12, 26, 59)$ \\ %& \footnotesize$( 8, 15, 38, 43)$ \\ %& \footnotesize$( 7, 19, 24, 54)$ \\ %& \footnotesize$(11, 13, 21, 59)$ \\ %

& \footnotesize$(5, 8, 38, 53)$ \\ %& \footnotesize$(10, 13, 19, 62)$ \\ %& \footnotesize$(13, 21, 32, 38)$ \\ %& \footnotesize$( 8, 19, 23, 54)$ \\ %& \footnotesize$( 9, 29, 31, 35)$ \\ %

\hline \footnotesize $D=$ 107

& \footnotesize$(3, 5, 13, 87)$ \\ %& \footnotesize$(34, 36, 67, 78)$ \\ %& \footnotesize$( 4, 43, 83, 85)$ \\ %& \footnotesize$( 8, 18, 33, 49)$ \\ %& \footnotesize$( 6, 16, 27, 59)$ \\ %

& \footnotesize$(3, 5, 38, 62)$ \\ %& \footnotesize$(15, 23, 34, 36)$ \\ %& \footnotesize$( 9, 43, 78, 85)$ \\ %& \footnotesize$( 4, 14, 31, 59)$ \\ %& \footnotesize$(12, 19, 27, 50)$ \\ %

& \footnotesize$(3, 5, 43, 57)$ \\ %& \footnotesize$(34, 36, 57, 88)$ \\ %& \footnotesize$(10, 43, 77, 85)$ \\ %& \footnotesize$( 5, 36, 82, 92)$ \\ %& \footnotesize$( 3, 45, 75, 92)$ \\ %

& \footnotesize$(3, 8, 14, 83)$ \\ %& \footnotesize$( 8, 31, 33, 36)$ \\ %& \footnotesize$( 3, 13, 25, 67)$ \\ %& \footnotesize$(17, 23, 30, 38)$ \\ %& \footnotesize$(36, 49, 63, 67)$ \\ %

& \footnotesize$(3, 14, 26, 65)$ \\ %& \footnotesize$(14, 27, 31, 36)$ \\ %& \footnotesize$( 3, 23, 38, 44)$ \\ %& \footnotesize$( 4, 51, 70, 90)$ \\ %& \footnotesize$(21, 23, 28, 36)$ \\ %

& \footnotesize$(3, 14, 41, 50)$ \\ %& \footnotesize$(19, 22, 31, 36)$ \\ %& \footnotesize$(20, 23, 27, 38)$ \\ %& \footnotesize$( 4, 47, 73, 91)$ \\ %& \footnotesize$( 4, 15, 27, 62)$ \\ %

& \footnotesize$(4, 7, 39, 58)$ \\ %& \footnotesize$(17, 25, 27, 39)$ \\ %& \footnotesize$( 7, 25, 30, 46)$ \\ %& \footnotesize$( 4, 11, 30, 63)$ \\ %& \footnotesize$(11, 24, 27, 46)$ \\ %

& \footnotesize$( 4, 63, 73, -32)$ \\ %& \footnotesize$( 8, 11, 27, 62)$ \\ %& \footnotesize$( 9, 17, 39, 43)$ \\ %& \footnotesize$( 5, 19, 22, 62)$ \\ %& \footnotesize$(10, 12, 19, 67)$ \\ %

\hline \footnotesize $D=$ 109 %4, símplices

& \footnotesize$(3, 5, 29, 73)$ \\ %& \footnotesize$(12, 63, 71, 73)$ \\ %& \footnotesize$(16, 22, 29, 43)$ \\ %& \footnotesize$( 5, 45, 75, 94)$ \\ %& \footnotesize$(3, 22, 94, 100)$ \\ %

& \footnotesize$(3, 5, 44, 58)$ \\ %& \footnotesize$(17, 58, 71, 73)$ \\ %& \footnotesize$(13, 22, 32, 43)$ \\ %& \footnotesize$(42, 47, 57, 73)$ \\ %& \footnotesize$( 3, 47, 77, 92)$ \\ %

& \footnotesize$(3, 40, 82, -15)$ \\ %& \footnotesize$(5, 9, 23, 73)$ \\ %& \footnotesize$(14, 19, 30, 47)$ \\ %& \footnotesize$( 4, 58, 60, 97)$ \\ %& \footnotesize$(20, 22, 29, 39)$ \\ %

& \footnotesize$(3, 41, 81, -15)$ \\ %& \footnotesize$(5, 59, 73, 82)$ \\ %& \footnotesize$(6, 8, 11, 85)$ \\ %& \footnotesize$(10, 22, 29, 49)$ \\ %& \footnotesize$( 4, 35, 89, 91)$ \\ %

\hline \footnotesize $D=$ 113 %2, símplices:

& \footnotesize$(3, 5, 38, 68)$ \\ %& \footnotesize$(15, 25, 36, 38)$ \\ %& \footnotesize$( 9, 15, 22, 68)$ \\ %& \footnotesize$(3, 22, 98, 104)$ \\ %& \footnotesize$(5, 36, 88, 98)$ \\ %

& \footnotesize$( 3, 17, 32, 62)$ \\ %& \footnotesize$(17, 27, 32, 38)$ \\ %& \footnotesize$(3, 20, 38, 53)$ \\ %& \footnotesize$(3, 53, 67, 104)$ \\ %& \footnotesize$(20, 25, 31, 38)$ \\ %

\hline \footnotesize $D=$ 119 %2, símplices

& \footnotesize$(3, 5, 38, 74)$ \\ %& \footnotesize$(15, 27, 38, 40)$ \\ %& \footnotesize$( 9, 24, 40, 47)$ \\ %& \footnotesize$( 3, 47, 92, 97)$ \\ %& \footnotesize$( 8, 22, 37, 53)$ \\ %

& \footnotesize$(3, 8, 47, 62)$ \\ %& \footnotesize$(19, 24, 37, 40)$ \\ %& \footnotesize$( 9, 15, 22, 74)$ \\ %& \footnotesize$( 5, 24, 38, 53)$ \\ %& \footnotesize$( 5, 48, 92, 94)$ \\ %

\hline \footnotesize $D=$ 121 %%1, símplice
& \footnotesize$(3, 8, 14, 97)$ \\ %& \footnotesize$(8, 76, 78, 81)$ \\ %& \footnotesize$(3, 45, 89, 106)$ \\ %& \footnotesize$(19, 26, 34, 43)$ \\ %& \footnotesize$(5, 51, 81, 106)$ \\ %

\hline \footnotesize $D=$ 125
& \footnotesize$(3, 8, 14, 101)$ \\ %& \footnotesize$( 8, 37, 39, 42)$ \\ %& \footnotesize$(3, 47, 92, 109)$ \\ %& \footnotesize$(9, 53, 91, 98)$ \\ %& \footnotesize$(11, 26, 42, 47)$ \\ %

\hline \footnotesize $D=$ 127
& \footnotesize$(3, 5, 53, 67)$ \\ %& \footnotesize$(20, 67, 83, 85)$ \\ %& \footnotesize$(12, 51, 91, 101)$ \\ %& \footnotesize$(12, 67, 85, 91)$ \\ %& \footnotesize$(3, 53, 91, 108)$ \\ %

\hline \footnotesize $D=$ 137
& \footnotesize$(3, 5, 58, 72)$ \\ %& \footnotesize$(44, 46, 72, 113)$ \\ %& \footnotesize$(13, 55, 98, 109)$ \\ %& \footnotesize$(7, 26, 46, 59)$ \\ %& \footnotesize$(3, 59, 97, 116)$ \\ %

\hline \footnotesize $D=$ 139
& \footnotesize$(3, 5, 59, 73)$ \\ %& \footnotesize$(22, 73, 91, 93)$ \\ %& \footnotesize$(16, 28, 41, 55)$ \\ %& \footnotesize$(33, 40, 93, 113)$ \\ %& \footnotesize$(3, 19, 40, 78)$ \\ %

\hline \footnotesize $D=$ 149 %2, símplices
& \footnotesize$(3, 5, 64, 78)$ \\ %& \footnotesize$(48, 50, 78, 123)$ \\ %& \footnotesize$(17, 30, 44, 59)$ \\ %& \footnotesize$(7, 50, 114, 128)$ \\ %& \footnotesize$(3, 63, 105, 128)$ \\ %

& \footnotesize$(3, 8, 14, 125)$ \\ %& \footnotesize$(8, 45, 47, 50)$ \\ %& \footnotesize$(3, 56, 110, 130)$ \\ %& \footnotesize$(23, 32, 42, 53)$ \\ %& \footnotesize$(13, 31, 50, 56)$ \\ %

\hline \footnotesize $D=$ 169 %%1, símplice

& \footnotesize$(3, 5, 74, 88)$ \\ %& \footnotesize$(27, 88, 111, 113)$ \\ %& \footnotesize$(19, 34, 50, 67)$ \\ %& \footnotesize$(16, 89, 113, 121)$ \\ %& \footnotesize$(3, 71, 121, 144)$ \\ %

\hline \footnotesize $D=$ 179 %%1, símplice
& \footnotesize$(3, 5, 79, 93)$ \\ %& \footnotesize$(58, 60, 93, 148)$ \\ %& \footnotesize$(20, 36, 53, 71)$ \\ %& \footnotesize$(9, 34, 60, 77)$ \\ %& \footnotesize$(3, 77, 127, 152)$ \\ %
\hline
\end{tabular}

\end{tabular}
\bigskip
\caption{The $179$ empty $4$-dimensional simplices of width $\ge 3$, where a vector $v\in \Z^4$ represents the simplex $\Delta(v):=\conv(e_1,e_2,e_3,e_4,v)$, of determinant $D=\sum v_i-1$.
All of them have width three except the simplex $\Delta( 6, 14, 17, 65)$ (highlighted), of width four and $D=101$}
\label{table:full_list}
\end{table}

Our proof of Theorem~\ref{thm:main} combines a theoretical part and a computational part:

\begin{itemize}
\item In Section \ref{sec:4hollow} we show that:

\begin{theorem}
\label{thm:bound_intro}
There is no empty $4$-simplex of width $3$ with determinant greater than $5058$.
\end{theorem}

\item In Section~\ref{sec:enumeration} we describe an enumeration procedure through which we have verified that:
\begin{theorem}
\label{thm:eunumeration_intro}
Among the empty $4$-simplices of volume at most $7600$ the only ones of width larger than two are the 179 found by Haase and Ziegler.
\end{theorem}
\end{itemize}

To prove Theorem~\ref{thm:bound_intro} we divide empty $4$-simplices into two classes: those that project (by an integer affine functional) to a \emph{hollow} $3$-polytope and those that do not. Here, a lattice polytope or a more general convex body is called \emph{hollow} (or \emph{lattice-free}) if it does not have any  lattice points in its interior. 

\begin{itemize}
\item If $P$ is an empty simplex of width larger than three that projects to a hollow $3$-polytope $Q$ then $Q$ itself has width larger than three. There are only five hollow $3$-polytopes of width three or larger, all of width three, as classified by~\cite{AKW}. By looking at them we can easily conclude that both $Q$ and $P$ have normalized volume bounded by 27 (Proposition~\ref{prop:hollow_projection}).

\item If $P$ is an empty (or, more generally, hollow) simplex of width larger than three that does not project to a hollow $3$-polytope, then we consider the projection of $P$ along the direction where its rational diameter (as defined in Section~\ref{sec:convex_geom}) is attained. 
%Recent results of~\cite{AKN} \sout{(see our Lemma~\ref{lemma:ca})} 
%\marginpar{Lemma~\ref{lemma:ca} is wrong\\ but the conclusion is\\ correct. See appen-\\dix}
%\paco{define rational diameter?}
We show that the rational diameter of $P$ is at most $42$ (Theorem~\ref{thm:bound}(1)) and combining this with several inequalities on successive minima and covering minima we are able to prove that $P$ must have volume bounded by $5058$ (Theorem~\ref{thm:bound}(2)). A key step for this, to which we devote the whole Section~\ref{sec:3hollow}, is an upper bound on the volume of hollow $3$-dimensional convex bodies of width larger than $1+2\sqrt{3}$ (Theorem~\ref{thm:Santos_I} and Corollary~\ref{coro:Santos_I}). These results generalize (and their proofs use ideas from)~\cite[Prop.~11]{AKW}.
\end{itemize}

For the computational part (Theorem~\ref{thm:eunumeration_intro}) we use the fact that all empty $4$-simplices are cyclic, which allows us to represent an empty simplex $P$ of determinant $D$ as a quintuple $(u_0,\dots,u_4)\in \T_D^4$. Here $\T_D^4$ is the discrete torus of order $D$ (in homogeneous coordinates); that is:
\[
\T_D^4:= \{(u_0,\dots,u_4) \in {\Z_D}^5 : \sum u_i = 0 \pmod D\}.
\]
Recall that saying that $P$ is cyclic means that the quotient group $\Z^4/\Lambda(P)$ is cyclic, where $\Lambda(P)$ is the lattice generated by vertices of $P$. In these conditions, if $(x_0,\dots, x_4)$ are the barycentric coordinates with respect to $P$ of a generator of this quotient group $\Z^4/\Lambda(P)$, we have that $(x_0,\dots, x_4) \in \frac{1}{D} \Z^5$ and we take $D\cdot (x_0,\dots, x_4)\in \Z^5$ as the quintuple representing $P$. Quintuples differing by integer multiples of $D$ represent the same simplex, so that there is a finite number of quintuples to check for each $D$ (a priori $D^4$, since one of the entries is determined by the other four, although in practice we only need to check $D^3$ by reasons to be mentioned in Section~\ref{sec:enumeration}).
Observe that the quintuple corresponding to a simplex $\Delta(v_1,v_2,v_3,v_4)$, in the notation of Table~\ref{table:full_list}, is $(-1, v_1,v_2,v_3,v_4)$.

This formalism makes it rather easy to do computations with cyclic simplices. In particular, in Section~\ref{sec:enumeration} we describe two algorithms that compute all empty $4$-simplices of a given determinant $D$. They work for different values of $D$; one requires $D$ not to be a prime power and the other requires $D$ not to have more than four different prime factors. We have run one or the other algorithm for every $D\in \{1,\dots,7600\}$ thus obtaining a computational proof of Theorem~\ref{thm:eunumeration_intro}. The total computation time was about 10\,000 hours of CPU.

\subsubsection*{Acknowledgements}
Most of the calculations in this work were done with support from the Santander Supercomputing  Group at the University of Cantabria, which provided access to the  Altamira Supercomputer at the Institute of Physics of Cantabria (IFCA-CSIC), a member of the Spanish Supercomputing Network.

We want to thank Benjamin Nill for his invitation to Magdeburg in April 2016 and for helpful discussions and clarifications regarding references~\cite{AKN, AKW, Perr08, Wessels}.

\section{Maximum volume of wide hollow $3$-polytopes}
\label{sec:3hollow}

In this section we give an upper bound for the volume of a $3$-dimensional convex body $K$ in terms of its width, assumed big (see Theorem~\ref{thm:Santos_I} and Corollary~\ref{coro:Santos_I}). We do not need to assume $K$ to be a lattice polytope, or even a polytope.

\subsection{Quick review of convex geometry tools}
\label{sec:convex_geom}

For the statement and proof we need to discuss \emph{successive minima} and \emph{covering minima} of convex bodies with respect to a lattice $\mathcal{L}$. Remember that:

\begin{enumerate}
\item For a centrally symmetric convex body $C\subset \R^d$, 
 the $i$-th successive minimum ($i\in\{1,\dots,d\}$) of $C$ with respect to $\mathcal{L}$ is:
\[
\lambda_i(C):=\inf\{\lambda>0 : \dim(\lambda C\cap \mathcal{L})\geq i\}.
\]
That is to say, $\lambda_i$ is the minimum dilation factor such that $\lambda C$ contains $i$ linearly independent lattice vectors.
Clearly $\lambda_1\le \dots \le \lambda_d$. 

\item  For a (not necessarily symmetric) convex body $K\subset \R^d$, 
 the $i$-th covering minimum ($i\in\{1,\dots,d\}$) of $K$ with respect to $\mathcal{L}$ is defined as
\[
\mu_i(K):=\inf\{\mu>0 : \mu K + \mathcal{L} \text{ intersects every affine subspace of dimension } d-i\}.
\]
Clearly $\mu_1\le \dots \le \mu_d$. 
\end{enumerate}

For example, $\mu_1(K)$ is nothing but the reciprocal of the lattice width of $K$, while $\mu_d(K)$ equals the covering radius of $K$ (the minimum dilation factor $\mu$ such that $\mu K + \mathcal{L}$ covers $\R^d$). 
Similarly, $\lambda_1(C)$ equals twice the packing radius of $C$ (the maximum dilation such that $\lambda C$ does not overlap any lattice translation of it).

Minkowski's Second Theorem~\cite{Gru93} 
relates successive minima and volume of a centrally symmetric convex body as follows:
\begin{equation}
\lambda_1(C)\lambda_2(C)\cdots\lambda_{d}(C)\vol(C)\leq 2^{d}.
\label{eq:Minkowski2}
\end{equation}

Successive minima are not usually defined for a non-centrally symmetric body $K$ (but see~\cite{HHH16}), 
but in this case the successive minima of the \emph{difference body} $K-K:=\{x-y:x,y\in K\}$ have a natural geometric meaning. For example, $\lambda_1(K-K)^{-1}$ equals 
the maximum (lattice) length of a rational segment contained in $K$. 
We call this the \emph{rational diameter of $K$}.\footnote{Observe that this is not the same as the ``lattice diameter" used in \cite{AKN}, defined as the maximum length of a \emph{lattice} segment contained in $K$.}
The volume of the difference body $K-K$ is bounded from below and from above by the Brunn-Minkowski and the Rogers-Shephard inequalities \cite{Gru93}, respectively:
\begin{equation}
2^d \vol(K) \le \vol(K-K) \le {2d \choose d} \vol (K).
\label{eq:voulme_of_difference}
\end{equation}
The lower bound (resp.,~the upper bound)  is an equality if and only if $K$ is centrally symmetric (resp.,~a simplex).

From this we can derive the following inequality relating the volume of $K$ and the first successive minimum of its difference body:
\begin{equation}
\vol(K)   \le   \frac{\vol(K-K) }{ 2^d}   
\le   \frac{1}{\prod_{i=1}^d \lambda_i(K-K)}
\le   \frac{1}{ \lambda_1(K-K)^d}.
\label{eq:First}
\end{equation}

Less is known about the covering radii,
but the following inequalities relating covering minima of $K$ and successive minimum of $K-K$ are
known (\cite{Kan88,Hur90} or see, e.g., \cite[Section 4]{AKW}):
\begin{eqnarray}
\label{eq:AKW1}
\mu_{i+1}(K)-\mu_i(K) &\le& \lambda_{d-i}(K-K), \qquad \forall i\in\{1,\dots,d-1\},\\
\label{eq:AKW2}
\mu_2(K)  &\le& \left( 1 + \frac{2}{\sqrt{3}}\right) \mu_1(K).
\end{eqnarray}

\subsection{The general case}

With these tools we can now prove the main result in this section. Both our statement and proof are based on~\cite[Proposition 11]{AKW}, which is the case $w=3$. The theorem can also be considered the three-dimensional version of~\cite[Thm. 4.1]{AWW}, which gives bounds for the volume of a convex $2$-body of width larger than 1.

\begin{theorem}
	\label{thm:Santos_I}
	Let $K$ be a hollow convex $3$-body of lattice width $w$, with
	$w > 1+2/\sqrt{3}=2.155$ and let $\mu=w^{-1}$ be its first covering minimum.
	Then, 

		$\vol(K) $ is bounded above by
	\begin{eqnarray*}
	&&\frac8{(1-\mu)^3} =\frac{8w^3}{(w-1)^3}, \qquad
		  \text{if } w\ge \frac{2}{\sqrt{3}}(\sqrt{5}-1) +1 = 2.427, \text{ and}\\
	&&\frac{3}{{4{\mu}^2(1 - \mu(1+2/\sqrt{3}))}} = \frac{3w^3}{4(w - (1+2/\sqrt{3}))},\qquad
		   \text{if } w\le  2.427.
	\end{eqnarray*}
		
\end{theorem}

Figure~\ref{fig:plot1} plots the upper bound of Theorem~\ref{thm:Santos_I} in the interval $w\in[2.4,5]$ that will be of interest for us.

\begin{figure}[ht]
\centerline{
\includegraphics[scale=0.9]{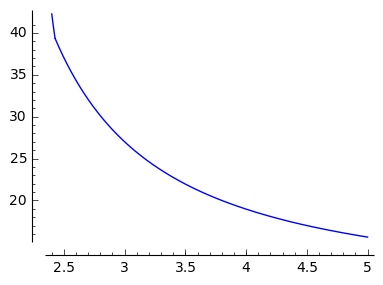}}
\caption{The upper bound of Theorem~\ref{thm:Santos_I} for the Euclidean volume of a hollow $3$-body (Y axis) in terms of its width (X axis).} 
\label{fig:plot1}
\end{figure}

The upper bound in this theorem is certainly not tight, but the threshold $w > 1+2/\sqrt{3}=2.155$ is. Since there is a hollow (non-lattice) triangle of width $ 1+2/\sqrt{3}$ (see~\cite[Figure 2]{Hur90}), taking prisms over it we can construct hollow $3$-polytopes of arbitrarily large volume. (In fact, inequality~\eqref{eq:AKW2} is equivalent to the statement ``no hollow polygon can have lattice width larger than $1+2/\sqrt{3}$''; the triangle in~\cite[Figure 2]{Hur90} shows that this is tight).

\begin{remark}
Theorem~\ref{thm:Santos_I} may be a step towards computing the exact value of the ``flatness constant'' (the maximum width among hollow bodies) in dimension three. By~\cite{AKW} no hollow $3$-polytope has width larger than three, but hollow polytopes with vertices not in the lattice can have larger width.
\end{remark}

\begin{proof}
	Throughout the proof we denote $\mu_i=\mu_i(K)$ and $\lambda_i=\lambda_i(K-K)$.

	We use the following slightly modified version of Equation~\eqref{eq:First}:
	\[
	\vol(K) \le \frac{1}{2^d} \vol(K-K)\le (\lambda_1\lambda_2\lambda_3)^{-1}\le \max\{{\lambda_1}^3, \lambda_1{\lambda_2}^2\}^{-1}.
	\]
	Our goal is to lower bound $\max\{{\lambda_1}^3, \lambda_1{\lambda_2}^2\}$. For this we combine
	equations \eqref{eq:AKW1} as follows:
	\[
	\lambda_2 \ge \mu_2-\mu_1  \ge \mu_3-\mu_1 -\lambda_1 \ge 1-\mu_1-\lambda_1,
	\]
	where $\mu_3\ge 1$ since $K$ is hollow.

	That is:
	\[
	\lambda_1\lambda_2\lambda_3\ge \max \{\lambda_1^3, \lambda_1\lambda_2^2\}
	\ge \max \{\lambda_1^3, \lambda_1(1-\mu_1-\lambda_1)^2\}.
	\]
	
	There are the following possibilities for the maximum on the right:
	\begin{itemize}
		\item If $\lambda_1 \ge (1-\mu_1)/2$ then either $1-\lambda_1-\mu_1$ is negative (and than smaller then $\lambda_1$ in absolute value, since $1-\mu_1$ is positive) or positive but smaller than $\lambda_1$. In both cases the maximum is  $\lambda_1^3$, which in turn is at least $(1-\mu_1)^3/8$.
		
		\item If $\lambda_1\le (1-\mu_1)/2$ then $1-\lambda_1-\mu_1$ is positive and bigger than $\lambda_1$, so the maximum is  $\lambda_1(1-\lambda_1-\mu_1)^2$. Now, by equations \ref{eq:AKW1} and \ref{eq:AKW2} we have
		\[
		\lambda_1 \ge \mu_3 - \mu_2 \ge 1 - (1+2/\sqrt{3})\mu_1,
		\]
		we take as lower bound for $\lambda_1(1-\lambda_1-\mu_1)^2$ the absolute minimum of $f(\lambda):=\lambda(1-\lambda-\mu_1)^2$ in the interval
		\[
		1 - (1+2/\sqrt{3}) \mu_1 \le  \lambda \le (1-\mu_1)/2.
		\]
		Since the only local minimum of $f$ is in $\lambda=1-\mu_1$, which is outside the interval, the minimum is achieved at one of the extremes. That is,
		\[
		f(\lambda) \ge \min\left\{ \left(1 - (1+2/\sqrt{3})\mu_1\right)\frac{4{\mu_1}^2}{3}, (1-\mu_1)^3/8\right \}.
		\]
		Now there are two things to take into account:
		\begin{itemize}
			\item For the interval to be non-empty we need
			\[
			1 - \mu_1(1+2/\sqrt{3})\le (1-\mu_1)/2,
			\]
			which is equivalent to
			\[
			{\mu_1}^{-1}\le 1+4/\sqrt{3}=3.31.
			\]
			\item Whenever ${\mu_1}^{-1}$ is between $\frac{2}{\sqrt{3}}(\sqrt{5}-1) +1 = 2.427$ and $1+4/\sqrt{3}=3.31$ we have
			\[
			\min\left\{ \left(1 - \mu_1(1+2/\sqrt{3})\right)\frac{4{\mu_1}^2}{3}, (1-\mu_1)^3/8\right\} = (1-\mu_1)^3/8,
			\]
			while for ${\mu_1}^{-1}\le \frac{2}{\sqrt{3}}(\sqrt{5}-1) +1 = 2.427$ the minimum is $(1 - \mu_1(1+2/\sqrt{3}))\frac{4{\mu_1}^2}{3}$.
		\end{itemize}
	\end{itemize}
\end{proof}

\subsection{The case of 5 points}
\label{sec:radon}

Observe that in the proof of Theorem~\ref{thm:Santos_I} what we bound is actually $\vol(K-K)$, and then use te Brunn-Minkowski inequality $\vol(K) \le \vol(K-K)/8$ to transfer the bound to $\vol(K)$. This means that with additional information on $K$ a sharper bound can be obtained. For example, if $K$ is a tetrahedron then we know $\vol(K) = \vol(K-K)/20$, which means that the bounds in the statement can all be multiplied by a factor of $8/20 = 0.4$. The case of interest to us in Section~\ref{sec:4hollow} is the somewhat similar case where $Q$ is a $3$-dimensional polytope expressed as the convex hull of five points (that is, either a tetrahedron, a square pyramid, or a triangular bipyramid). We now analyze this case in detail.

Most of what we want to say about this case is valid in arbitrary dimension, so let $A=\{a_1,\dots,a_{d+2}\}\subset \R^{d}$ 
be $d+2$ points affinely spanning $\R^d$. In particular, there is a unique (modulo a scalar factor)  vector $\lambda=(\lambda_1,\dots,\lambda_{d+2})$ such that
\[
\sum_{i=1}^{d+2} \lambda_i a_i =0.
\]
This naturally partitions $A$ into three subsets (of which only $A^0$ can be empty):
\[
A^+:=\{a_i : \lambda_i >0\},\quad 
A^0:=\{a_i : \lambda_i =0\}, \quad
A^-:=\{a_i : \lambda_i <0\}.
\]
Of course, $A^+$ and $A^-$ are interchanged when multiplying $\lambda$ by a negative constant, but $A^0$ and the partition of $A\setminus A^0$ into two parts are independent of the choice of $\lambda$. In fact:
\begin{enumerate}
\item $a_i\in A^0$ if, and only if, $A\setminus \{a_i\}$ is affinely dependent (equivalently, $A$ is a pyramid over $A\setminus \{a_i\}$).
\item $(A^+,A^-)$ is the only partition of $A\setminus A^0$ into two parts such that $\conv(A^+) \cap $ $\conv(A^-) \ne \emptyset$.
\item In fact, $\conv(A^+) \conv(A^-)$ is a single point. It is the unique point of $\R^d$ that can be expressed as a convex combination of each of two disjoint subsets of $A$. We call this point the \emph{Radon point} of $A$, since its existence and the partition of $A$ into three parts is basically Radon's theorem~\cite{Ziegler}.
\end{enumerate}

Observe that both $\conv(A^-)$ and $\conv(A^0\cup A^+)$ are simplices, by property (1) above, and that their affine spans are complementary: their dimensions add up to $d$ and they intersect only in the Radon point. By an affine change of coordinates, we can make the Radon point to be the origin, and the affine subspace containing  $\conv(A^-)$ and $\conv(A^0\cup A^+)$ be complementary coordinate subspaces. In these conditions, $\conv(A)$ is the \emph{direct sum} of $\conv(A^-)$ and $\conv(A^0\cup A^+)$, where the direct sum of polytopes $P\subset \R^p$ and $Q\subset \R^q$  both containing the origin is defined as 
\[
P\oplus Q:= \conv(P\times \{0\} \cup \{0\}\times Q) \subset \R^{p+q}.
\]

Since the volume of a direct sum has the following simple formula
\[
\vol(P\oplus Q) = \frac{\vol(P) \vol(Q)}{{p+q \choose p}},
\]
we have the following result:

\begin{lemma}
\label{lemma:Radon}
Let $A\subset \R^{d}$ be a set of $d+2$ points affinely spanning $\R^d$, and let $p=|A^-|-1$ and $q=|A^0\cup A^+|-1$, so that $p+q=d$. Let $K=\conv(A)$. Then:
\[
\vol(K-K) \ge {2p\choose p}{2q \choose q}\vol(K).
\]
\end{lemma}

\begin{proof}
By an affine transformation, let $A^-\subset \R^p\times \{0\}$ and $A^0\cup A^+ \subset \{0\}\times \R^q$, and let $P\subset \R^p$ and $Q\subset \R^q$ be the corresponding convex hulls, which are simplices of dimensions $p$ and $q$ respectively. By the Rogers-Shephard inequality:
\[
\vol(P-P) = {2p\choose p}\vol(P), 
\qquad
\vol(Q-Q) = {2q\choose q}\vol(Q).
\]
Now, $K=P\oplus Q$ implies
\[
K-K= (P\oplus Q) - (P\oplus Q) \supseteq (P-P) \oplus (Q-Q). 
\]
In particular, $\vol(K-K)$ is at least
\[
\vol((P-P) \oplus (Q-Q)) = \frac{\vol(P-P) \vol(Q-Q)}{{p+q \choose p}} =
\frac{{2p\choose p} {2q\choose q}}{{p+q \choose p}}\vol(P)\vol(Q).
\]
On the other hand,
\[
\vol(K) = \vol(P\oplus Q) = \frac{\vol(P) \vol(Q)}{{p+q \choose p}}.
\]
\end{proof}

\begin{corollary}
\label{coro:Radon}
Let $K$ be the convex hull of five points affinely spanning $\R^3$. Then
\[
\vol(K-K) \ge 12\vol(K).
\]
\end{corollary}

\begin{proof}
Since the $p$ and $q$ in Lemma~\ref{lemma:Radon} are non-negative and add up to three, there are only two possibilities: $(p,q)\in \{(0,3),(3,0)\}$ or  $(p,q)\in \{(1,2),(2,1)\}$.  The lemma gives 
$\vol(K-K) \ge 20\vol(K)$
and
$\vol(K-K) \ge 12\vol(K)$
respectively.
\end{proof}

We do not expect the factor $12$ in the statement of Corollary~\ref{coro:Radon} to be sharp, but it is not far from sharp; if $K$ is a pyramid with square base then $\vol(K-K) = 14 \vol(K)$

\begin{corollary}
	\label{coro:Santos_I}
	Let $K$ be the convex hull of five points affinely spanning $\R^3$, and assume it to be hollow.
	Let $w > 1+2/\sqrt{3}=2.155$  be the width of $K$ and let $\mu=w^{-1}$ be its first covering minimum.
	Then, 

		$\vol(K) $ is bounded above by
	\begin{eqnarray*}
	&&\frac{16}{3(1-\mu)^3} =\frac{16w^3}{3(w-1)^3}, \qquad
		  \text{if } w\ge \frac{2}{\sqrt{3}}(\sqrt{5}-1) +1 = 2.427, \text{ and}\\
	&&\frac{1}{{2{\mu}^2(1 - \mu(1+2/\sqrt{3}))}} = \frac{w^3}{2(w - (1+2/\sqrt{3}))},\qquad
		   \text{if } w\le  2.427.
	\end{eqnarray*}
	
\end{corollary}

\begin{proof}
With the same notation as in the proof of Theorem~\ref{thm:Santos_I}, thanks to Lemma~\ref{lemma:Radon} we have
	\[
	\vol(K) \le \frac{1}{12} \vol(K-K)\le\frac23 (\lambda_1\lambda_2\lambda_3)^{-1}\le \frac23\max\{{\lambda_1}^3, \lambda_1{\lambda_2}^2\}^{-1}.
	\]
	In order to lower bound $\max\{{\lambda_1}^3, \lambda_1{\lambda_2}^2\}$ follow word by word the proof of Theorem~\ref{thm:Santos_I}.
\end{proof}

\section{Maximum volume of wide hollow $4$-simplices}
\label{sec:4hollow}

Throughout this section it is convenient to measure volumes \emph{normalized to the lattice}. That is, we use $\Vol(P):=d!\vol(P)$ for any $P\subset \R^d$, where $\vol$ denotes the Euclidean volume and $\Vol$ the normalized one. This makes all lattice polytopes to have integer volume, and the volume of a simplex be equal to its determinant.

Here we give an upper bound on the determinant (equivalently, the volume) of any hollow $4$-simplex of width at least three.
Our main idea is to consider an integer projection $\pi:P \rightarrow Q\subset\R^3$ and transfer to $P$ the bound for the volume of $Q$ that we have in Corollary~\ref{coro:Santos_I}. Observe that $Q$ is the convex hull of $5$ points and it has width at least three (because any affine integer functional on $Q$ can be lifted to $P$, with the same width) but it will not necessarily be hollow. Thus, some extra work is needed. 
A road map to the proof is as follows:

\begin{itemize}
\item If a projection $\pi$ exists for which $Q$ is hollow, then $Q$ is a hollow polytope of width at least three. Such polytopes have been classified in \cite{AKW,AWW}: there are only five, with maximum volume $27$. It is easy to prove (looking at the five possibilities) an upper bound of $27$ for the determinant of the simplex $P$. See details in Proposition~\ref{prop:hollow_projection}.
\item If such a $\pi$ does not exist, then we show that 
%a result of Averkov et al.~\cite{AKN} implies 
that ${\lambda_1}^{-1}(P-P) \le 42$ (see part (1) of Theorem~\ref{thm:bound}, based on Corollary~\ref{coro:onepoint}).
%\marginpar{Lemma~\ref{lemma:ca} is wrong\\ but the conclusion is\\ correct. See appen-\\dix}
We then have a dichotomy:
\begin{itemize}
\item If $Q$ is ``close to hollow'' (that is, if it contains a hollow polytope of about the same width) then we can still use Corollary~\ref{coro:Santos_I} to get a good bound on its volume, hence on the volume of $P$.
\item If $Q$ is ``far from hollow''  (that is, if it has interior lattice points far from the boundary) then 
it is easy to get much better bounds on ${\lambda_1}^{-1}(P-P)$, which by Minkowski's Theorem directly give us a bound on the volume of $P$.
\end{itemize}
\end{itemize}

\subsection{$P$ has a hollow projection $Q$}

We start with a general result about projections of a simplex to codimension one. Observe that if $P\subset\R^d$ is a (perhaps not-lattice) $d$-simplex and $\pi:P\to Q\subset \R^{d-1}$ is a projection of it then $Q=\pi(P)=\subset \R^{d-1}$ can be written as the convex hull of $d+1$ points, the images under $\pi$ of the vertices of $P$. In this situation we call \emph{Radon point of $Q$} the Radon point of $\pi(\vertices(P))$,  introduced in Section~\ref{sec:radon}.

\begin{lemma}
\label{lemma:radon}
Let $\pi:P\to Q\subset \R^{d-1}$ be an integer projection of a $d$-simplex $P$. 
Let $x\in Q$ be the Radon point of $Q$ and let $s=\pi^{-1}(x)\subset P$ be the fiber of $x$ in $P$ (a segment).
Then:
\begin{enumerate}
\item $\Vol(P) = \Vol(Q) \times \length(s)$, where $\length(s)$ is the lattice length of $s$.
\item $s$ maximizes the lattice length among all segments in $P$ in the projection direction.
\end{enumerate}
\end{lemma}

\begin{proof}
Observe that every facet $F$ of $Q$ not containing $x$ is a simplex (because its vertices are affinely independent), and that $\pi$ is a bijection from $\pi^{-1}(F)$ to $F$. Let us consider $Q$ triangulated by coning $x$ to each of those facets. Let $S=\conv(F\cup\{x\})$ be one of the maximal simplices in this triangulation. Then, $\pi^{-1}(S)$ is also a simplex, with one vertex projecting to each vertex of $F$ and the segment $s$ projecting to $x$. This implies part (2) of the statement, and also the following analogue of part (1) for Euclidean (as opposed to normalized) volumes:
\[
\vol(P) = \vol(Q) \times \length(s) / d.
\]
From this, $\Vol(P) = d! \vol(P)$ and $\Vol(Q) = (d-1)! \vol(Q)$ gives part (1).
\end{proof}

\begin{proposition}
\label{prop:hollow_projection}
If an empty $4$-simplex $P$ of width at least three can be projected to a hollow lattice $3$-polytope $Q$, then the normalized volume of $P$ is at most $27$.
\end{proposition}

\begin{proof}
$Q$ is one of the five hollow lattice $3$-polytopes of width at least three, classified in~\cite{AKW,BlancoHaaseHoffmanSantos}. Figure~\ref{fig:max_hollow} shows the five of them.
\begin{figure}[htb]
\includegraphics[scale=.9]{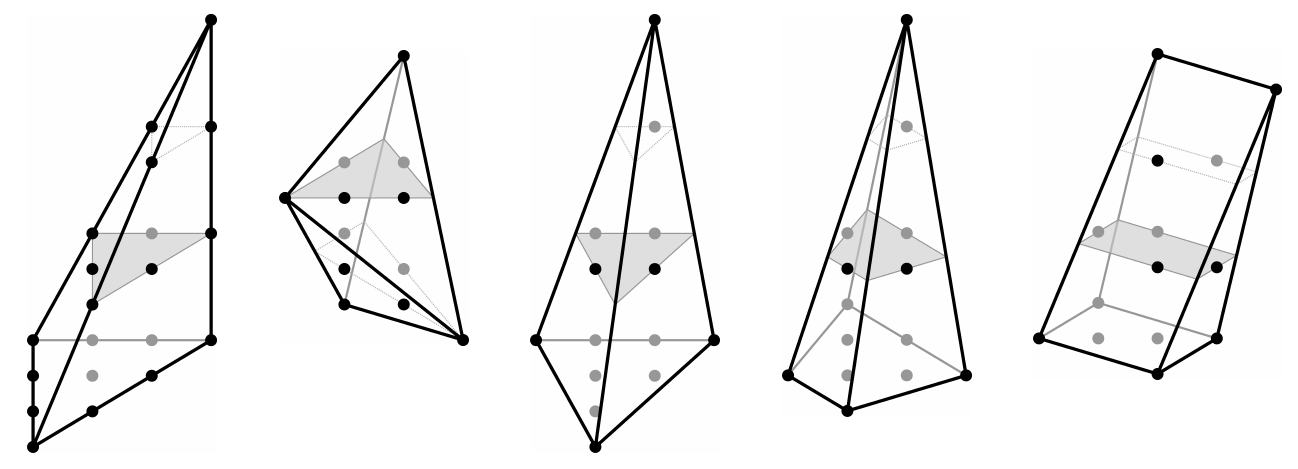}
\centerline{$
{Q}_{1} \qquad\qquad\quad
{Q}_{2} \qquad\qquad\quad
{Q}_{3} \qquad\qquad\quad
{Q}_{4} \qquad\qquad\quad
{Q}_{5} \quad
$}
\caption{The five hollow $3$-polytopes of width three. Figure taken from~\cite{AKW}}
\label{fig:max_hollow}
\end{figure}

Their normalized volumes are $27$, $25$, $27$, $27$, and $27$, respectively.
${Q}_{5}$ cannot be the projection of $P$, since it has six vertices. For the other four, we claim that the Radon point $x$ of $\pi(\vertices(P))$ is always a lattice point:

\begin{itemize}

\item In the three tetrahedra $Q_1$, $Q_2$ and $Q_3$ this is automatic: four vertices of $P$ project to the four vertices of the tetrahedron $Q$ and the fifth vertex projects to a lattice point in $Q$ which necessarily equals the Radon point.

\item In $Q_4$, a pyramid over a quadrilateral, the Radon point is the intersection of the two diagonals of the quadrilateral, which happens to be a lattice point.

\end{itemize}

Now, by Lemma~\ref{lemma:radon}, $\Vol(P) = \Vol(Q) \times \length(s)$, where $\length(s)$ is the lattice length of the fiber of the Radon point. Since the Radon point is a lattice point and since $P$ is hollow, the length of this fiber is at most $1$. On the other hand, $\Vol(Q)$ is at most $27$.
\end{proof}

\subsection{$P$ has no hollow projection}

Our first tool is Corollary~\ref{coro:onepoint} which guarantees that for every non-hollow lattice simplex $T$ and facet $F$ of $T$ there is a lattice point in the interior of $T$ not too close to $F$. The lower bound obtained is expressed in terms of the Sylvester sequence $(s_i)_{i\in \N}$. In particular, 
for dimension $4$ our bound is ${1}/(s_4-1)=1/{42}$.

\begin{lemma}
\label{lemma:onepoint}
Let $T$ be a simplex and let $S\subset T$ be a finite set of points including the vertices of $T$ and at least one point in the interior of $T$.

Then, for each vertex $a$ of $T$ there is a subsimplex $T'\subset T$ with exactly one point of $S$ in its interior and with $a\in \vertices(T')\subset S$.
\end{lemma}

Here $T$ and $T'$ are not assumed to be full-dimensional, or to have the same dimension. In particular, by ``interior''  we mean the relative interior. They are also not assumed to be lattice simplices unless $S$ is a set of lattice points (which is the case of interest to us).

\begin{proof}
If $T$ has a unique point of $S$ in its interior then there is nothing to prove, simply take $T'=T$ for every $a$.
If $T$ has more than one such point we argue by induction on the number of them. 
%Let $F$ be the facet of $T$ opposite to $F$

Let $y\in \text{int}(T)\cap S$ be an interior  point minimizing the barycentric coordinate with respect to $a$ and let $\mathcal T$ be the stellar triangulation of $T$ from $y$. (The maximal simplices in $\mathcal T$ are $\conv(F\cup \{y\})$ for the facets $F$ of $T$). Let $y'\in S$ be another interior  point in $T$ and let $T'$ be the minimal simplex in $\mathcal T$ containing $y'$. Then:
\begin{itemize}
\item By minimality of $T'$, $y'$ is in the interior of $T'$. In particular, $T'$ is not hollow.
\item $T'$ uses $a$ as a vertex, since $T'$ contains the interior point $y$ and all simplices of $\mathcal T$ that do not contain $a$ are contained in the boundary of $T$.
\item By construction, $\text{int}(T') \subset \text{int}(T)$ (remark: $T'$ may be not full-dimensional; by $\text{int}(T') $ we mean the relative interior). Since $y$ is an interior point  in $T$ but not in $T'$, $T'$ has less interior points than $T$ and we can apply the induction hypothesis to it.
\end{itemize}

\begin{figure}[ht]
\centerline{
\includegraphics[scale=0.9]{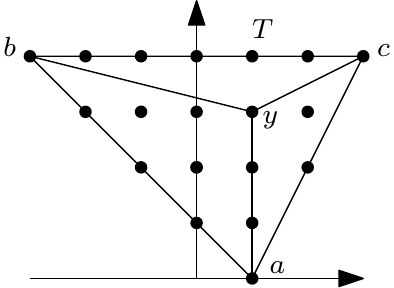}}
\caption{Illustration of Lema~\ref{lemma:onepoint}. Once the point $y$ is chosen, the other six points of $S$ (in the figure, a lattice point set) in the interior of $T$ are valid choices for $y'$. Depending of the choice of $y'$ we get a different $T'$: $T'=\{a,b,y\}$ if $y'$ is one of $(-1,3)$, $(0,3)$ or $(0,2)$; $T'=\{a,y\}$ if $y'$ is one of  $(1,1)$ or $(1,2)$; and $T'=\{a,c,y\}$ if $y'$ is $(2,3)$. Choice of $y$ guarantees that no interior point of $S$ lies in $\{y,b,c\}$.} 
\label{fig:onepoint}
\end{figure}

\end{proof}

\begin{corollary}
\label{coro:onepoint}
Let $T$ be a non-hollow lattice $d$-simplex and let $a$ be a vertex of it. Then, there is an interior lattice point in $T$ whose barycentric coordinate with respect to $a$ is at least $1/(s_{d+1}-1)$.
\end{corollary}

\begin{proof}
Let $S=T\cap \Z^d$ and  let $T'$ be as in Lemma~\ref{lemma:onepoint}. Observe that $\dim(T')\le \dim(T)$ so that  $s_{d'+1}\le s_{d+1}$. Since $T'$ has a unique interior lattice point, the statement is the case $i=d+1$ of \cite[Theorem 2.1]{AKN}.
\end{proof}

Our next result makes precise what we mean by ``far from hollow'' in the description at the beginning of this section, and how to use that property to upper bound the rational diameter of $P$ (and hence its volume, via Lemma~\ref{lemma:radon} and Corollary~\ref{coro:Santos_I}).

\begin{lemma}
\label{lemma:lambda_vs_r}
Let $P\subset\R^d$ be a hollow convex body. Consider an integer projection $\pi:P\to Q\subset \R^{d-1}$ of $P$.
Let $x\in Q$ be an arbitrary point and let $s_x=\pi^{-1}(x)\subset P$ be its fiber.

Then, $\length(s_x)^{-1} \ge 1-r$, where $\length(s)$ is the lattice length of $s$.
\end{lemma}

\begin{proof}
Observe there if $Q$ is hollow then $r\ge1$, so the statement is trivial. Thus, without loss of generality we assume $Q$ is not hollow. Also, if $x$ is a lattice point in the interior of $Q$ then $r=0$ and the fact that $P$ is hollow implies $\length(s_x)\ge 1$, so the statement holds. Thus, we assume that $Q$ has at least an interior lattice point that is not $x$.

Let $y\in Q\setminus \{x\}$ be an interior lattice point of $Q$ closest to $x$ with respect to the seminorm induced by $Q$ with center at $x$. That is, for each interior lattice point $p$ in $Q$ call $||p||_{Q,x}$ the smallest dilation factor $r_p$ such that $p\in x+r_p(Q-x)$, and let $y$ be a lattice point minimizing that quantity. Observe that $r=||y||_{Q,x}$. (Remark: we do not assume $x$ to be in the interior of $Q$. If $Q$ lies in the boundary the seminorm $||p||_{Q,x}$ may be infinite for points outside $Q$, but it is always finite and smaller than one for points in the interior).

Let $s_y=\pi^{-1}(y)\subset P$ be the fiber of $y$. The length of $s_y$ must be at most $1$, because $P$ is hollow. 
Consider the ray from $x$ through $y$ and let $z$ be the point where it hits the boundary of $Q$. We have:
\[
\length(s_x) \le \frac{\length(s_x)}{\length(s_y)} \le \frac{|xz| }{ |yz|},
\]
where the second inequality follows from convexity of $P$.
Then:
\[
\length(s_x)^{-1} \geq \frac{|yz|}{|xz|} =\frac{|xz|-|xy|}{|xz|}= 1- \frac{|xy|}{|xz|} = 1-||y||_{Q,x} = 1-r.
\]
\end{proof}

With this we can prove our main result. In it, we consider the projection $\pi:P\to Q$  along the direction where $\lambda_1(P-P)$ is achieved. 
Equivalently, this is the direction of the longest (with respect to the lattice) rational segment contained in $P$. 
(Recall that this length is called the \emph{rational diameter} of $P$, and it equals 
$\lambda_1(P-P)^{-1}$). In particular, if we let the point $x$ of Lemma~\ref{lemma:lambda_vs_r} be the point whose fiber achieves the rational diameter, we have $\lambda_1(P-P)\ge 1-r$. Moreover, Lemma~\ref{lemma:radon} tells us that $x$ is the Radon point of $Q$ and that  $\Vol(P) = \Vol(Q)/\lambda_1(P-P)$.

\begin{theorem}
\label{thm:bound}
Let $P$ be a hollow $4$-simplex of width at least three and that does not project to a hollow $3$-polytope. Then:
\begin{enumerate}
\item $\lambda_1(P-P)\ge 1/42$. 
\item $\Vol(P) \le 5058$. 
\end{enumerate}
%Every hollow lattice $4$-simplex of width at least $3$ and determinant greater than $5058$ projects to a hollow $3$-polytope.
%In particular, there is no empty lattice $4$-simplex of width at least $3$ and determinant greater than $5058$.
\end{theorem}

%\begin{lemma}
%\label{lemma:parche}
%Let $P$ be a hollow $4$-simplex that does not project to a hollow $3$-polytope. Then $\lambda_1(P-P)\ge 1/42$. 
%\end{lemma}

\begin{proof}
Let $Q$ be the projection of $P$ along the direction where  $\lambda_1(P-P)$ is attained. We know $Q$ is not hollow, and has width at least three. For the rest of the proof we denote $\lambda:=\lambda_1(P-P)$. 

Suppose first that $\lambda\ge 0.19$, in which case part (1) obviously holds. For part (2) we can simply bound the volume of $P$ in the manner of Equation~\eqref{eq:First}, except that for a $d$-simplex $P$ we can use the Rogers-Shephard equality, see Eq.~\eqref{eq:voulme_of_difference}:
\[
\vol(P-P) = {2d \choose d} \vol(P).
\]
Thus:
\[
\Vol(P) = 24 \vol(P) = \frac{24}{70} \vol(P-P) 
  \le \frac{24\cdot 16}{70 {\lambda}^4} 
  \le \frac{24\cdot 16}{70 \cdot {0.19}^4} \le 4209.38.
\]

So, {\it for the rest of the proof we assume $\lambda<0.19$.} 

Let $x$ be the Radon point of $Q$, that is, the image of the segment where $\lambda$ is achieved. Let $r$ be as in Lemma~\ref{lemma:lambda_vs_r}, so that $r\ge 1 -\lambda\geq 0.81$
and $Q_r:=x + r(Q-x)$ is hollow. Together with Lemma~\ref{lemma:radon} this gives:

\begin{equation}
\Vol(Q)
%=\frac{\Vol(Q)}{\lambda}
=\frac{\Vol(Q_r)}{r^3}
\le\frac{6\vol(Q_r)}{(1-\lambda)^3}.
\label{eq:Theorem}
\end{equation}

Observe that the width of $Q_r$ is $r$ times the width of $Q$ and, in particular, 
$w(Q_r)\geq 3\cdot 0.81= 2.43$. Since $Q_r$ is the convex hull of five points (the projection of the five vertices of $P$), Corollary~\ref{coro:Santos_I} gives
\begin{eqnarray}
\label{eq:Theorem2}
\vol(Q_r) \le 
\frac{16}{3(1-\mu_1(Q_r))^3} %= 
%\frac{16}{3\left(1-\frac{\mu_1(Q)}r\right)^3}
\le \frac{16}{3\left(1-\frac1{3(1-\lambda)}\right)^3}
=  \frac{2^4 3^2(1-\lambda)^3}{(2-3\lambda)^3},
\end{eqnarray}
where the inequality in the middle follows from $\mu_1(Q)\le 1/3$ ($Q$ has width at least three) and
$\mu_1(Q_r) =  \mu_1(Q)/r \le \mu_1(Q)/(1-\lambda)$.
% and $r\ge 1-\lambda$.

Consider $Q$ triangulated centrally from the Radon point $x$. That is, for each facet $F$ of $Q$ not containing $x$  we consider the tetrahedron $\conv(F\cup \{x\})$. (Observe that all facets of $Q$ not containing $x$ are triangles, since the only affine dependence among the vertices of $Q$ is precisely the one that gives the Radon point). We call such tetrahedra the \emph{Radon tetrahedra} in $Q$.

Let $v_1,\dots,v_5$ be the five vertices of $Q$. (Remark: if $Q$ has only four vertices then the Radon point is the projection of the fifth vertex of $P$. But the Radon point being integer implies $\lambda\ge 1$). For each $i\in \{1,\dots,5\}$ denote $Q_i:= \conv(\{v_1,\dots, v_5\}\setminus \{v_i\})$ the lattice tetrahedron contained in $Q$ and not using vertex $v_i$.

There are two possibilities:

\begin{itemize}
\item Suppose first that one of the $Q_i$'s has the following property: only one of the Radon tetrahedra contained in that $Q_i$ contains interior lattice points of $Q$. Let $T:=\conv(F\cup \{x\})$ be that Radon tetrahedron. 
%Observe that the points are necessarily in the interior of $\conv(F\cup \{x\})$ (otherwise they would be shared with another tetrahedron).
Let $T':=\conv(F'\cup \{x\})$ be a minimal face of $T$ that contains some interior lattice point $y$ of $Q$. Then:
Minimality of $F'$ implies that $y$ is in the interior of $T'$, by minimality.

Let $v_j\not\in T$ be the vertex of $Q_i$ not in $F$ and let $Q'_i= \conv(F'\cup \{v_j\})$. 
Observe that $\conv(F'\cup \{x\})\subset \conv(F'\cup \{v_j\})$ (because the Radon point of $Q$ is contained in every simplex spanned by vertices of $Q$ and intersecting the interior of $Q$). 

By Corollary~\ref{coro:onepoint}  the non-hollow lattice simplex $\conv(F'\cup \{v_j\})$ contains an interior point $z$ whose barycentric coordinate with respect to the facet $F'$ is at least $1/42$. This is the same as the barycentric coordinate of $z$ in $T$ with respect to the facet $F$.
Now,  by uniqueness of $\conv(F\cup \{x\})$ as a Radon tetrahedron in $T$, $y$ is also contained in $\conv(F\cup \{x\})$. Moreover,  the barycentric coordinate of $y$ in $\conv(F\cup \{x\})$, which is a lower bound for $\lambda$ by the same arguments as in Lemma~\ref{lemma:lambda_vs_r}, is greater than in $\conv(F\cup \{a\})$. Thus, $\lambda \ge 1/42$.

\item Suppose now that every $Q_i$ contains either zero or at least two Radon tetrahedra with interior lattice points of $Q$. An easy case study shows that then at least four Radon tetrahedra contain interior lattice points of $Q$. (The cases are that $Q$ is a triangular bipyramid with six Radon tetrahedra or a quadrangular pyramid with four Radon tetrahedra).
Let $\conv(F_i\cup \{x\})$  be such Radon tetrahedra, and let $y_i\in \conv(F_i\cup \{x\})$ be an interior lattice point of $Q$, for $i\in \{1,2,3,4\}$. (Some of the $y_i$'s may coincide, since we do not assume them to be interior in $\conv(F_i\cup \{x\})$, only in $Q$).
%are at least two facets $F_1$ and $F_2$ such that $\conv(F_i\cup \{x\})$ contains and interior point $y_i$ of $Q$ (the interior point may be the same or different).
Then, $1/\lambda$ is smaller than $\frac{\Vol(\conv(F_i\cup \{x\}))}{\Vol(\conv(F_i\cup \{y_i\}))}$, for each $i$, and this is smaller than $\Vol(\conv(F_i\cup \{x\}))$ since $\conv(F_i\cup \{y_i\})$ is a lattice tetrahedron. That is:
\[
\Vol(Q) \ge \sum_{i=1}^4 \Vol(\conv(F_i\cup \{x\}))  \ge \frac4\lambda.
\]
This together with equations \eqref{eq:Theorem} and \eqref{eq:Theorem2} gives%
\[
 \frac4\lambda \le \Vol(Q) \le  \frac{2^5 3^3}{(2-3\lambda)^3},
\]
which implies $(\frac23-\lambda)^3\le 8 \lambda$ or, equivalently, $\lambda \ge 0.03196 > 1/42$. 
\end{itemize}

So, in both cases we have $\lambda\ge 1/42$, which finishes the proof of part (1). 
%\end{proof}

%\begin{theorem}
%\label{thm:bound}
%Every hollow lattice $4$-simplex of width at least $3$ and determinant greater than $5058$ projects to a hollow $3$-polytope.
%%In particular, there is no empty lattice $4$-simplex of width at least $3$ and determinant greater than $5058$.
%\end{theorem}

%\marginpar{This proof is incom-\\plete, since it uses\\ Lemma~\ref{lemma:ca}. See\\ a corrected proof in\\ the appendix}

%\begin{proof}

%Let $P$ be a hollow $4$-simplex that does not project to a hollow $3$-polytope. Let
% $Q$ be its projection through the direction where the $\lambda_1(P-P)$ is attained. We call 
%$\lambda:=\lambda_1(P-P)$. We know $Q$ is not hollow.
%
%
%If $\lambda \ge 0.19$, then we can simply bound the volume of $P$ in the manner of Equation~\eqref{eq:First}, except that for a $d$-simplex $P$ we can use the Rogers-Shephard equality, see Eq.~\eqref{eq:voulme_of_difference}:
%
%\[
%\vol(P-P) = {2d \choose d} \vol(P).
%\]
%Thus:
%\[
%\Vol(P) = 24 \vol(P) = \frac{24}{70} \vol(P-P) 
%  \le \frac{24\cdot 16}{70 {\lambda}^4} 
%  \le \frac{24\cdot 16}{70 \cdot {0.19}^4} \le 4209.38.
%\]

For part (2),  Equations~\eqref{eq:Theorem} and~\eqref{eq:Theorem2} give
%
%
%Let $x$ be the image of the segment where the rational diameter of $P$ is achieved and let $r$ be as in Lemma~\ref{lemma:lambda_vs_r}, so that $r\ge 1 -\lambda\geq 0.81$
%and $Q_r:=x + r(Q-x)$ is hollow. Together with Lemma~\ref{lemma:radon} this gives:
%
%\begin{equation}
%\Vol(P)
%=\frac{\Vol(Q)}{\lambda}
%=\frac{\Vol(Q_r)}{r^3\lambda}
%\le\frac{6\vol(Q_r)}{(1-\lambda)^3\lambda}.
%\label{eq:Theorem3}
%\end{equation}
%
%
%Observe that the width of $Q_r$ is $r$ times the width of $Q$ and, in particular, 
%$w(Q_r)\geq 3\cdot 0.81= 2.43$. Since $Q_r$ is the convex hull of five points (the projection of the five vertices of $P$), Corollary~\ref{coro:Santos_I} gives
%\[
%\vol(Q_r) \le \frac{16}{3(1-\mu_1(Q_r))^3} 
%= \frac{16}{3\left(1-\frac{\mu_1(Q)}r\right)^3}
%\le \frac{16}{3\left(1-\frac1{3(1-\lambda)}\right)^3}
%=  \frac{2^4 3^2(1-\lambda)^3}{(2-3\lambda)^3},
%\]
%where the inequality in the middle follows from $\mu_1(Q)\le 1/3$ ($Q$ has width at least three) and $r\ge 1-\lambda$.
%
%Putting this together with Equation~\eqref{eq:Theorem} gives
%
\begin{equation}
\label{eq:vol_lambda}
\Vol(P) = \frac{\Vol(Q)}{\lambda} \le \frac{2^53^3}{(2-3\lambda)^3\lambda}.
\end{equation}

\begin{figure}[ht]
\centerline{
\includegraphics[scale=0.6]{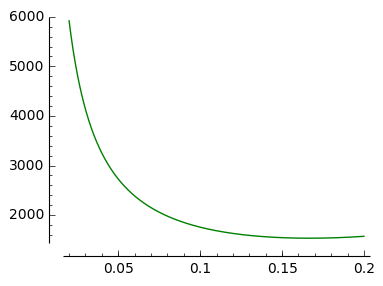}}
\caption{Plot of  the upper bound in Eq.~\eqref{eq:vol_lambda}
%$f(\lambda)=\frac{2^53^3}{(2-3\lambda)^3\lambda}$ 
for $\lambda\in [0.02,0.20]$.}
\label{fig:plot2}
\end{figure}

%Now, \sout{by Lemma~\ref{lemma:ca} we know that} $\lambda \ge 1/42$.
%\marginpar{Lemma~\ref{lemma:ca} is wrong\\ but the conclusion is\\ correct. See appen-\\dix}
Figure~\ref{fig:plot2} plots this function in the relevant range $\lambda\in [1/42, 0.19]$. Although the function is slightly increasing after its local minimum at $\lambda=1/6$, its maximum in the interval is clearly at $\lambda= 1/42$, where it takes the value
\[
\Vol(P)\le 
\frac{2^5\,3^3\, 42}{(2-1/14)^3} = 
\frac{14^3\,2^5\,3^3\, 42}{27^3}=
\frac{2^9\,7^4}{3^5}= 5058.897.
\]
We can round this down to $5058$ since $\Vol(P)\in \Z$.
%
%\chn
\end{proof}

%\paco{(The statement of) this corollary is a bit awkward}
%\begin{corollary}
%There is no empty $4$-simplex of width greater than $4$, and there is only one of width $4$. And it is %described explicitly in the list of Haase and Ziegler.
%\end{corollary}

\section{Enumeration of empty $4$-simplices for a given determinant}
\label{sec:enumeration}

Let $P$ be a lattice simplex $P\in \R^d$, and let $\Lambda(P)$ be the lattice generated by vertices of $P$. We assume with no loss of generality that the origin is a vertex of $P$, so that $\Lambda(P)$ is a linear lattice and $G(P):=\Z^d/\Lambda(P)$ is a finite group of order equal to the determinant of $P$. One way to store $P$ is via generators of $G(P)$ as a subgroup of $\R^d/\Lambda(P)$, with barycentric coordinates used in $\R^d$. Let us be more precise:

\begin{itemize}
\item The barycentric coordinates of a point $x\in \R^d$ with respect to the simplex $P$ is the vector $(x_0,\dots,x_d)$ of coefficients
of the unique expression of $x$ as an affine combination of the vertices of $P$ (the vertices of $P$ are assumed given in a particular order). They add up to one; conversely, any $(d+1)$-vector with real coefficients and sum of coordinates equal to one represents a unique point in $\R^d$ in barycentric coordinates. If $x$ is a lattice point and $P$ a lattice simplex of determinant $D$, then all the $x_i$'s lie in $\frac{1}{D} \Z$.

\item Looking at $x$ in the quotient $\R^d/\Lambda(P)$ is equivalent to looking at the $x_i$'s modulo $\Z$; that is, looking only at the fractional part of them. In particular, every lattice point $u\in \Z^d$, considered as an element of the quotient $\Z^d/\Lambda \subset \R^d/\Lambda$, can be represented as a vector $(u_0,\dots,u_d) \in (\Z_D)^{d+1}$ with sum of coefficients equal to $0$ modulo $D$.

\item In this manner, to every lattice simplex $P$ of determinant $D$ we associate a subgroup $G(P)$ of order $D$ of 
the group
\[
\T_D^d :=\{(u_0,\dots,u_d)\in \Z_D^{d+1} : \sum u_i=0 \pmod D\}.
\]
\end{itemize}

We call $\T_D^d$ the discrete $d$-torus of order $D$, since it is isomorphic to $(\Z/D\Z)^d$.
In this setting we have:

\begin{lemma}
\label{lemma:subgroups}
Let $G_1$ and $G_2$ be two subgroups of order $D$ of $\T_D^d$. Then, $G_1$ and $G_2$ represent equivalent simplices of determinant $D$ if, and only if, they are the same subgroup modulo permutation of coordinates.
\end{lemma}

\begin{proof}
The ``if'' part is obvious. 
For the ``only if'' observe that a unimodular equivalence $f: P_1 \to P_2$ between two lattice simplices preserves barycentric coordinates, modulo the permutation of vertices induced by $f$. 
That is, if $(x_0,\dots,x_d)$ are the barycentric coordinates of a point $x$ with respect to $P_1$, then the barycentric coordinates of $f(x)$ with respect to $P_2$ are a permutation of them: the
$i$-th barycentric coordinate of $f(x)$ with respect to $P_2$ equals $x_j$, where $j$ and $i$ are such that the $f$ maps the $j$-th vertex of $P_1$ to the $i$-th vertex of $P_2$.
\end{proof}

This formalism is specially useful if $P$ is a cyclic simplex, that is, if $G(P)$ is a cyclic group. In this case we represent $G(P)$ by giving a generator of it. Every primitive element $(u_0,\dots,u_d) \in \T_D^d$ generates such a cyclic group of order $D$ and, in particular, a cyclic lattice $d$-simplex of order $D$. (Here, we call $(u_0,\dots,u_d)$ primitive if $\gcd(u_0,\dots,u_d)=1$).
Observe that this includes all empty $4$-simplices, since they are all cyclic (Barile et al.~\cite{BBBK09}, see part (4) of Theorem~\ref{thm:intro_previous_results}). Hence, we introduce the following definition:

\begin{definition}
\label{def:quintuples}
Let $P$ be a cyclic lattice $4$-simplex of determinant $D$ and let $u\in \T_D^4$. We say that \emph{the quintuple $u$ generates $P$} if the barycentric coordinates of every element in $\Z^4/\Lambda(P)$ with respect to $P$ are multiples of $u$ (modulo $D$). That is, if the point of $\R^4/\Lambda(P)$ with barycentric coordinates $\frac{1}{D}u$ is a generator for the cyclic group $\Z^4/\Lambda(P)$.
\end{definition}

Observe the the entries of $\frac{1}{D}u$ may not add up to $1$, but they add up to an integer. This is what they need to satisfy in order to represent a point of $\R^4/\Lambda(P)$, since the quotient by $\Lambda(P)$ makes barycentric coordinates be defined only modulo the integers.

Summing up: every primitive element $u\in \T_D^4$ is the quintuple of a cyclic lattice simplex of determinant $D$. Moreover:

\begin{itemize}
\item Two quintuples $u,v\in \T_D^4$ generate equivalent simplices if, and only if, one is obtained from the other one by permutation of entries and/or multiplication by a scalar coprime with $D$.

\item The width of $P$ equals the minimum $k\in \N$ such that there are $\lambda_0,\dots \lambda_4 \in \{0,1,\dots,k\}$ (not all zero) with $\sum_i \lambda_i u_i = 0 \pmod D$.

\item The element $ku$ of $\T_D$ represents a lattice point in $P$ if, and only if, when writing it with all entries in $\{0,\dots,D-1\}$ the sum of entries equals $D$. (Indeed, this means that the lattice point of $\R^d$ whose barycentric coordinates are $\frac{1}{D} ku$ is a convex combination of the vertices of $P$. In order to check whether $P$ is empty one can check that this does not happen for any $k=1,2,\dots, D-1$. 
\end{itemize}

With this in mind, we have the following two algorithms for computing all empty $4$-simplices of determinant $D$, depending on the value of $D$. In both we obtain a list of perhaps-non-empty simplices that contains all the empty ones. We then prune this list via the emptiness test mentioned above:

\subsection*{Algorithm 1 (works if $D$ has less than five prime factors)}
If $D$ has less than five prime factors, then it must have a unimodular facet, by the following lemma:

\begin{lemma}
\label{lemma:coprime_facets}
Let $D_1, \dots,D_5$ be the volumes of the five facets of an empty $4$-simplex $P$.
Then, $\gcd(D_i, D_j)=1$ for every $i,j$. In particular, $D$ is a multiple of $D_1\times\dots\times D_5$.
\end{lemma}

\begin{proof}
Let $\Lambda(P)\subset \Z^4$ be the lattice generated by the vertices of $P$ (to simplify things, assume one of the vertices to be the origin so that $\Lambda(P)$ is a linear lattice).
If a prime $p$ divides both $D_i$ and $D_j$ then there are lattice points $u$ and $v$ of order $p$ (as elements of $\Z^4/\Lambda(P)$) lying in the hyperplanes of the $i$-th and $j$-th facets $F_i$ and $F_j$ of $P$. These two elements of $\Z^4/\Lambda(P)$ cannot be multiples of one another, or otherwise the one that is a multiple would lie in the affine $2$-plane spanned by $F_i\cap F_j$ (which cannot happen since $F_i\cap F_j$ is unimodular, as is every empty triangle).

This implies that $u$ and $v$ generate a subgroup of $\Z^4/\Lambda(P)$ isomorphic to 
$\Z_p\times \Z_p$, in contradiction with the fact that $\Z^4/\Lambda(P)$ is cyclic.
(See part (4) of Theorem~\ref{thm:intro_previous_results}, taken from Barile et al.~\cite{BBBK09}).
\end{proof}

Then, every simplex of determinant $D$ is equivalent to one of the form 
\[
\Delta(v):=\conv(e_1,e_2,e_3,e_4,v)
\]
for a certain $v=(v_1,v_2,v_3,v_4)\in \Z^4$ with $v_1+v_2+v_3+v_4 = D+1$. (Equivalently, $\Delta(v)$ is generated by the quintuple $(-1,
v_1,v_2,v_3,v_4)$.  Algorithm 1 simply looks at all the quintuples of the form $(-1, v_1,v_2,v_3,v_4)\in \T_D^d$ one by one. Since $v_1,v_2,v_3,v_4\in \Z_D$ and $\sum v_i = 1 \pmod D$ this gives $D^3$ quintuples to test. Via symmetries and other considerations that quickly discard some quintuples corresponding to non-empty simplices the quintuples to be tested can be further reduced.

\subsection*{Algorithm 2 (works if $D$ has more than one prime factor; that is, $D$ is not a prime-power)}
In this case we can decompose $D=ab$ with $a$ and $b$ relatively prime. Observe that each quintuple $u=(u_0,\dots,u_d)$ that represents a cyclic simplex of determinant $D$ can be considered modulo $a$ and modulo $b$ to represent, respectively, a cyclic simplex $P_a$ of determinant $a$ and another $P_b$ of determinant $b$. 
We say that $P_a$ and $P_b$ are the relaxations of $P$ modulo $a$ and $b$, respectively.

We have the following result:

\begin{lemma}
\label{lemma:merge}
Let $P$ be a cyclic simplex of determinant $D=ab$ generated by the quintuple $(u_0,\dots,u_4)\in \T_D^4$. Then, the simplices $P_a$ and $P_b$ of determinants $a$ and $b$ generated by the same quintuple $(u_0,\dots,u_4)$ (via the natural maps $T_D^4\to \T_a^4$ and $\T_D^4 \to \T_b^4$) are empty as well.
\end{lemma}

\begin{proof}
The geometric meaning of considering the quintuple $(u_0,\dots,u_4)\in T_D^4$ as an element of $T_a^4$ for a divisor $a$ of $D$
 is to coarsen the lattice: we can think of $P_a$ as being the same simplex as $P$, but considered with respect to the lattice generated by the point of barycentric coordinates $\frac{1}{a} u$ instead of $\frac{1}{D} u$.
\end{proof}

Conversely, the fact that $\gcd(a,b)=1$ implies that from any pair of quintuples $(a_0,\dots,a_4)\in \T_a^4$ and $(b_0,\dots, b_4)\in T_b^4$ representing simplices $P_a$ and $P_b$ we can obtain a quintuple in $\T_D^4$ representing a simplex $P$ whose relaxations are $P_a$ and $P_b$. Simply let $u$ be a generator of the cyclic group (of order $D$) obtained as direct sum of the groups of $P_a$ and $P_b$ (of orders $a$ and $b$). For example, but not necessarily, take:
\[
(u_0,\dots,u_4) = b (a_0,\dots,a_4) + a(b_0,\dots,b_4).
\]
(Observe that multiplying the element $(a_0,\dots,a_4)\in \T_a^4$ by $b=D/a$ we naturally obtain an element of $\T_D^4$).

We have then our second algorithm { Algorithm 2} which goes as follows: 
\begin{enumerate}
\item Precompute all empty simplices of volumes $a$ and $b$.
\item For each empty simplices $P_a$ of determinant $a$ and $P_b$ of determinant $b$, consider a quintuple $(a_0,\dots,a_4)$ generating $P_a$ and all the quintuples generating simplices equivalent to $P_b$ (for this, start with one quintuple and compute all the ones obtained from it by permutation and/or multiplication by a scalar coprime with $b$).
\item For each pair of quintuples $(a_0,\dots,a_4)$ and $(b_0,\dots,b_4)$ so obtained consider the simplex of determinant $D$ generated by the quintuple $b(a_0,\dots,a_4) + a(b_0,\dots,b_4)$.
\end{enumerate}

The simplices obtained in step (3) are guaranteed to contain all empty simplices of volume $D$. As in Algorithm 1, some of them may not be empty so an emptiness test is needed to prune the list.

%\paco{donde damos los resultados? habria que enunciar el Teorema 1.5, o decir que los resultados son los de la tabla 1}
\subsection*{Computation times}
We have implemented the above algorithms in {\tt python} and run them in the  Altamira Supercomputer at the Institute of Physics of Cantabria (IFCA-CSIC) for every $D\in \{1,2,\dots, 7600\}$. These computations have proven Theorem~\ref{thm:eunumeration_intro}.

For many values of $D$ (those with two, three, or four prime factors) both algorithms work and we have chosen one of them. Also, for Algorithm 2 there is often several choices of how to split $D$ as a product of two coprime numbers $a$ and $b$.
Experimentally we have found that Algorithm 2 runs much faster if $a$ and $b$ are chosen of about the same size, and in this case it outperforms Algorithm 1. 
This is seen in Figure~\ref{fig:tiempos} where some computation times are plotted for the two algorithms. Blue points in the figure show the time taken for Algorithm 2 to compute all empty $4$-simplices for a given determinant of the form $D=2b$ with $b$ a prime number. Purple points correspond to the same computation for $D=ab$ with both $a$ and $b$ primes bigger than $12$. Green points are prime determinants, where Algorithm 1 needs to be used.
In Algorithm 2 the time to precompute empty simplices of determinants $a$ and $b$ is not taken into account, since we obviously had that already stored from the previous values of $D$. As seen in the figure, about $100\,000$ seconds (that is, about 1 day) computing time was needed in some cases with $D\sim 5000$. The total computation time for the whole set of values of $D$ was about $10\, 000$ hours ($\sim$1 year).

\begin{figure}[htb]
\includegraphics[scale=.5]{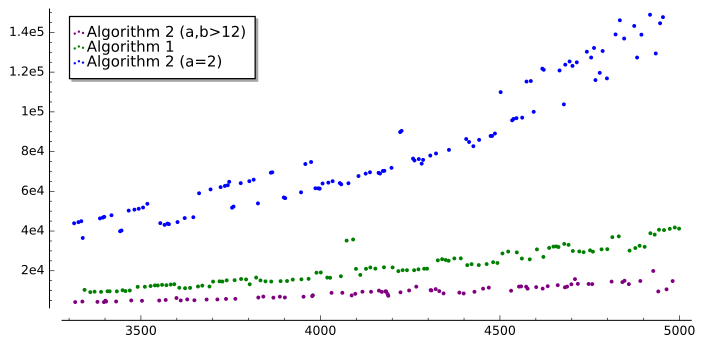}

\caption{Computation times (seconds) for the enumeration of empty $4$-simplices of a given determinant $D$ between $3000$ and $5000$}
\label{fig:tiempos}
\end{figure}

%\clearpage\input{parche.tex}

\end{document}